\documentclass[11pt, reqno]{amsart}

\usepackage[margin=1.3in]{geometry}
\usepackage{amsmath, amssymb, amsthm}
\usepackage{amsfonts, amscd, bm, cancel}
\usepackage{amsrefs}
\usepackage{xcolor}
\usepackage[colorlinks=true, urlcolor=blue, linkcolor=blue, citecolor=magenta, pdfstartview=FitHm]{hyperref}

\newtheorem{theorem}{Theorem}[section]
\newtheorem{lemma}[theorem]{Lemma}

\newtheorem{definition}[theorem]{Definition}
\newtheorem{corollary}[theorem]{Corollary}

\newtheorem{remark}[theorem]{Remark}

\newtheorem*{notation*}{Notation}

\numberwithin{equation}{section}
\allowdisplaybreaks
\begin{document}
\title[Fenchel-Willmore and Sobolev inequalities]{Fenchel-Willmore and Sobolev-type inequalities for submanifolds in non-negatively curved manifolds }

\author[M. Ji]{Meng Ji}
\address{School of Mathematics and Applied Statistics,
University of Wollongong\\
NSW 2522, Australia}
\email{\href{mailto:mengj@uow. edu.au}{mengj@uow.edu.au}}

\author[K.-K. Kwong]{Kwok-Kun Kwong}
\address{School of Mathematics and Applied Statistics,
University of Wollongong\\
NSW 2522, Australia}
\email{\href{mailto:kwongk@uow.edu.au}{kwongk@uow.edu.au}}



\begin{abstract}
In this paper, we uncover a novel connection between the Fenchel-Willmore inequality and a new logarithmic Sobolev inequality for mean-convex submanifolds immersed in non-negatively curved manifolds with Euclidean volume growth. Building on this connection, we establish extensions of the Fenchel-Willmore inequality to submanifolds with boundary and to complete non-compact submanifolds. In addition, we derive a sharp Sobolev-type inequality for submanifolds in the same setting. These Sobolev-type inequalities admit a number of applications, including topological consequences in the surface case.
\end{abstract}

\maketitle

\section{Introduction}
In recent decades, the study of geometric and analytic inequalities for submanifolds has become a prominent theme in differential geometry and geometric analysis. These inequalities provide powerful tools for connecting analytic properties with the intrinsic and extrinsic geometry of submanifolds, and they often serve as bridges between geometry, analysis, and topology.

Notable examples of such inequalities include the Sobolev and logarithmic Sobolev inequalities, which connect curvature conditions to analytic properties on submanifolds \cite{Brendle2022, Ecker2000, michael1973Sobolev}; the isoperimetric, Minkowski and Alexandrov-Fenchel inequalities, which relate surface area, enclosed volume, and (higher order) mean curvature integrals \cite{brendle2021isoperimetric, BrendleHungWang2016, GL09}; and the Fenchel-Willmore inequality, which provides a lower bound for an integral of a power of the mean curvature in terms of the ambient geometry \cite{AFM20, Chen1971, Fenchel29, Willmore68}.

There has been sustained effort devoted to taking inequalities such as those described above, originally proved in the Euclidean setting, and extending them to the Riemannian setting under curvature conditions.

For instance, in a recent paper \cite{JiKwong2025}, the authors proved that for a complete non-compact Riemannian manifold $(M, g)$ of dimension $n+m$ with nonnegative $k$-Ricci curvature (for some $k=k(m, n)$) and positive asymptotic volume ratio
$\displaystyle \theta:=\lim_{r \to \infty} \frac{|B(p, r)|}{|\mathbb{B}^{n+m}| r^{n+m}}$,
every closed $n$-dimensional submanifold $\Sigma$ immersed in $M$ satisfies the sharp Fenchel-Willmore inequality
\begin{equation}\label{ineq AFM}
\int_{\Sigma} |\sigma|^{n} \ge \theta |\mathbb{S}^{n}|,
\end{equation}
where $\sigma=\tfrac{1}{n}\, \mathrm{tr}\, \mathrm{II}$ is the mean curvature vector of $\Sigma$, with $\mathrm{II}$ denoting the second fundamental form. For more details, please see \cite[Theorem 1.3]{JiKwong2025}.
This result generalizes the relatively recent work of Agostiniani, Fogagnolo, and Mazzieri \cite{AFM20}, as well as the earlier foundational contributions of Chen \cite{Chen1971}, Fenchel \cite{Fenchel29}, and Willmore \cite{Willmore68}.

The inequality \eqref{ineq AFM} naturally gives rise to several further questions:\\
1. Can a version of the Fenchel-Willmore inequality be established for submanifolds with boundary?\\
2. What about submanifolds that are complete but non-compact?\\
3. Does the Fenchel-Willmore inequality admit an analytic counterpart? It is well known that the classical isoperimetric inequality is equivalent to the Sobolev inequality \cite[II.2]{Chavel2001}, highlighting the link between geometric and analytic aspects. In light of this, it is natural to ask whether the Fenchel-Willmore inequality has a similar interpretation in terms of an analytic inequality.

In the first part of this paper, we address the questions raised above by establishing new Fenchel-Willmore type inequalities for submanifolds that may be non-compact or may have boundary, and we further demonstrate that these inequalities can be derived from a logarithmic Sobolev type inequality for submanifolds.

Our first result extends the classical setting to compact submanifolds with boundary.

\begin{theorem}\label{thm fenchel willmore}
Let $n, m \in \mathbb{N}$, and $(M, g)$ be a complete non-compact Riemannian manifold of dimension $n+m$ with nonnegative sectional curvature and positive asymptotic volume ratio $\theta$. Suppose that $\Sigma$ is a compact $n$-dimensional submanifold immersed in $M$ (possibly with boundary $\partial \Sigma$) such that the normalized mean curvature vector $\sigma$ of $\Sigma$ is nowhere vanishing. Then
\begin{align}\label{ineq Fenchel-Willmore nonneg}
\int_{\Sigma}|\sigma|^n \ge \theta C_{n, m} e^{-\frac{|\partial \Sigma|}{\int_\Sigma |\sigma|}},
\end{align}
where
\begin{equation}\label{eq C nm}
C_{n, m}=\begin{cases}
\left|\mathbb{S}^n\right|\quad & \text{if }{m\le 3}\\
(n+1)\frac{\left|\mathbb{S}^{n+m-1}\right|}{\left|\mathbb{S}^{m-2}\right|}& \text{if }m{>} 3.
\end{cases}
\end{equation}
If {$m\le 3$}, the equality holds if and only if $\Sigma$ is {connected}, umbilical and with no boundary, and
$$
\int_{\Sigma} |\sigma|^n= \theta {|\mathbb S^n|}.
$$
\end{theorem}

In general, for a non-compact submanifold $\Sigma$, no non-trivial lower bound for $\int_{\Sigma}|\sigma|^n$ can be obtained without additional assumptions, since submanifolds that are minimal or close to minimal provide counterexamples. Despite this obstruction, we establish two new Fenchel-Willmore type inequalities for complete non-compact submanifolds. The first connects the Fenchel-Willmore integral $\int_{\Sigma}|\sigma|^n$ to the topology of $\Sigma$ via the Cohn-Vossen deficit, while the second links it to the isoperimetric constant of $\Sigma$. To the best of our knowledge, such inequalities have not been previously formulated, and they represent the first extensions of the Fenchel-Willmore inequality to the non-compact setting.

\begin{theorem}\label{thm non-compact fenchel willmore}
Let $(M^{2+m}, g)$ be as in Theorem \ref{thm fenchel willmore}, but now assume $\Sigma$ is a complete non-compact surface immersed in $M$, with $\sigma$ nowhere vanishing. Suppose
\begin{enumerate}
\item $\limsup_{r \to \infty} \frac{1}{r} \int_{B_r} |\sigma| = C > 0$, where $B_r$ is the metric ball of radius $r$ (by the induced metric on $\Sigma$) centered at a fixed point of $\Sigma$.
\item $\int_\Sigma K^- < \infty$, where $K^-$ is the negative part of the Gaussian curvature of $\Sigma$.
\end{enumerate}
Then both $\chi(\Sigma)$ and $\int_\Sigma K$ are finite, and
\begin{equation}\label{ineq non-compact willmore}
\int_{\Sigma} |\sigma|^2 \ge \theta C_{2, m} \exp\left(-\frac{2\pi \chi(\Sigma) - \int_\Sigma K}{C} \right),
\end{equation}
where $C_{2, m}$ is given in \eqref{eq C nm}.
\end{theorem}

The condition on $\frac{1}{r} \int_{B_r}|\sigma|$ rules out the possibility that $\Sigma$ is close to a minimal surface. Furthermore, by the Cohn-Vossen theorem \cite{CohnVossen1935}, the quantity $2 \pi \chi(\Sigma)-\int_{\Sigma} K$, called the Cohn-Vossen deficit, that appears on the RHS of \eqref{ineq non-compact willmore} is non-negative.

If $\Sigma$ is a higher-dimensional non-compact submanifold, one can instead obtain a lower bound for the Fenchel-Willmore integral $\int_\Sigma |\sigma|^n$ in terms of the isoperimetric constant of $\Sigma$ (Definition \ref{def isop const}), provided that this constant is sufficiently small. Interestingly, the proof relies on two types of Sobolev inequalities: a logarithmic Sobolev inequality established in this work, which implies Theorem \ref{thm fenchel willmore} (as explained below), and a Michael-Simon type Sobolev inequality \eqref{brendle2} due to Brendle.

\begin{theorem}\label{thm noncopct willmore higher dim}
Let $(M^{n+m}, g)$ be as in Theorem \ref{thm fenchel willmore}, but now assume $\Sigma$ is a complete $n$-dimensional submanifold immersed in $M$ ($n\ge2$), with $\sigma$ nowhere vanishing. If the isoperimetric constant of $\Sigma$ satisfies $C_{\mathrm {iso }}(\Sigma)<n\left(\theta K_{n, m}\right)^{\frac{1}{n}}$,
then
$$ \int_{\Sigma} |\sigma|^n \ge \theta C_{n, m} \exp\left(-\frac{n C_{\mathrm{iso}}(\Sigma)}{n\left(\theta K_{n, m}\right)^{\frac{1}{n}} -C_{\mathrm{iso}}(\Sigma)} \right). $$
Here, $K_{n, m} = \dfrac{|\mathbb{S}^{n+m-1}|}{|\mathbb{S}^{m-1}|}$ and $C_{n, m}$ is given in \eqref{eq C nm}.

In particular, if $C_{\mathrm{iso}}(\Sigma)=0$, then the classical Fenchel-Willmore inequality $\int_\Sigma|\sigma|^n \ge \theta C_{n, m}$ holds.
\end{theorem}

We now return to the earlier question of whether the Fenchel-Willmore inequality admits an analytic counterpart, in analogy with the classical equivalence between the isoperimetric and Sobolev inequalities. Remarkably, it turns out that the Fenchel-Willmore type inequality in Theorem \ref{thm fenchel willmore} can be derived as a corollary of a logarithmic Sobolev type inequality.
This shows a new connection between a Fenchel-Willmore type inequality and a logarithmic Sobolev inequality on submanifolds which involves the norm of the mean curvature vector.

Let us now state the sharp logarithmic Sobolev inequality that underlies this relation, which we establish for submanifolds (possibly with boundary) immersed in ambient manifolds of arbitrary codimension with nonnegative sectional curvature.

\begin{theorem}\label{thm log sob ineq}
Let $n, m \in \mathbb{N}$, and $(M, g)$ be a complete non-compact Riemannian manifold of dimension $n+m$ with nonnegative sectional curvature and asymptotic volume ratio $\theta>0$. Suppose that $\Sigma$ is a compact $n$-dimensional submanifold immersed in $M$ (possibly with boundary $\partial \Sigma$) such that the normalized mean curvature vector $\sigma$ of $\Sigma$ is nowhere vanishing. Let $f$ is a positive smooth function on $\Sigma$.
Then
\begin{equation}\label{log sob intro}
\int_{\Sigma} f|\sigma|\left(\log f+\log (\theta C_{n, m})\right)-\left(\int_{\Sigma} f|\sigma|\right) \log \left(\int_{\Sigma} f|\sigma|^n\right) \le \frac{n+1}{2 n^2} \int_{\Sigma} \frac{\left|\nabla^{\Sigma} f\right|^2}{f|\sigma|}+\int_{\partial \Sigma} f,
\end{equation}
where $C_{n, m}$ is given by \eqref{eq C nm}.
\end{theorem}
We have the following characterization of the equality case of \eqref{log sob intro}.
\begin{theorem}\label{thm log sob intro equality}
With the same assumptions and notations as in Theorem \ref{thm log sob ineq} and suppose that $m\le 3$, the equality in \eqref{log sob intro} holds if and only if
$f$ is constant, $\Sigma$ is umbilical and with no boundary, and
$$
\int_{\Sigma} |\sigma|^n= \theta |\mathbb{S}^{n}|.
$$
\end{theorem}
By choosing $f = 1$ in the logarithmic Sobolev inequality, we immediately recover the Fenchel-Willmore type inequality \eqref{ineq Fenchel-Willmore nonneg} for submanifolds with boundary of arbitrary codimension immersed in ambient spaces with nonnegative curvature. Except for the added assumption on the mean curvature vector, this provides the first extension of the Fenchel-Willmore inequality to settings where the submanifold is non-compact or has boundary.

This connection between the Fenchel-Willmore inequality and the logarithmic Sobolev inequalities for submanifolds is quite surprising to us.
Given this connection, we now briefly review some classical and recent developments in the theory of logarithmic Sobolev inequalities
with particular emphasis on results in the submanifold setting.

Logarithmic Sobolev inequalities (not restricted to the submanifold setting) form an important class of inequalities and have attracted significant attention due to their wide applicability. They have found applications, for instance, in entropy monotonicity along Ricci flow \cite{Perelman2002} and for the heat equation \cite{Ni2004}, monotonicity formula for mean curvature flow \cite{Huisken1990}, concentration of measure \cite{Ledoux2001}, and information theory \cite{Stam1959}.
The classical $L^p$-logarithmic Sobolev inequality in Euclidean space was first established by Gross \cite{Gross1975} and Weissler \cite{Weissler1978} for the case $p=2$, and later generalized to all $1<p<n$ by Del Pino and Dolbeault \cite{DelPinoDolbeault2003}.


In the setting of submanifolds in Euclidean space, the analysis of logarithmic Sobolev inequalities is more delicate. The first foundational contribution in this direction was due to Ecker \cite{Ecker2000}, who obtained a codimension-free but non-sharp $L^2$-logarithmic Sobolev inequality for Euclidean submanifolds. This was later sharpened by Brendle \cite{Brendle2022}, who established a codimension-free version with optimal constants, using techniques inspired by the Alexandrov-Bakelman-Pucci maximum principle \cite{cabre2008elliptic} and the optimal transport approach to the isoperimetric inequality \cite{BrendleEichmair2023}.
Later, following the strategies of Brendle, Pham \cite{Pham25} proved a sharp logarithmic Sobolev inequality for closed $n$-dimensional submanifolds $\Sigma$ in Riemannian manifolds ($M^{n+m}, g$) with non-negative sectional curvature, though under the additional assumption that the normalized mean curvature vector satisfies $|\sigma|=1$ everywhere on $\Sigma$, together with a characterization of the equality case.

%

His result can be stated as follows:
Let $n, m \in \mathbb{N}$, and let $(M, g)$ be a complete non-compact Riemannian manifold of dimension $n+m$ with nonnegative sectional curvature and asymptotic volume ratio $\theta > 0$. Suppose $\Sigma$ is a closed $n$-dimensional submanifold of $M$ with normalized mean curvature vector $\sigma$ satisfying $|\sigma| = 1$. Let $f$ be a positive smooth function on $\Sigma$. Then
\begin{equation*}
\int_{\Sigma} f\left(\log f + \log(\theta C_{n, m})\right) - \left(\int_{\Sigma} f\right) \log \left(\int_{\Sigma} f \right) \le \frac{n+1}{2 n^2} \int_{\Sigma} \frac{|\nabla^{\Sigma} f|^2}{f},
\end{equation*}
where $C_{n, m}$ is given by \eqref{eq C nm}.

Although Pham's inequality is sharp under the assumption that $|\sigma|$ is constant, this condition is restrictive. This provides another motivation for proving Theorem \ref{thm log sob ineq}, in addition to establishing the Fenchel-Willmore type inequality. Allowing $|\sigma|$ to vary broadens the class of submanifolds to which the inequality applies, thereby leading to the more general logarithmic Sobolev inequality in Theorem \ref{thm log sob ineq}.


In the second part of this paper, we establish another sharp Sobolev-type inequality for compact submanifolds immersed in ambient manifolds with nonnegative sectional curvature.

This is motivated by the results of Brendle \cite{brendle2021isoperimetric, Brendle2023}.
For a compact $n$-dimensional submanifold $\Sigma$ with (possibly non-empty) boundary in a non-negatively curved ambient manifold $M^{n+m}$, and for every smooth positive function $f$ on $\Sigma$, Brendle proved the inequality
\begin{equation}\label{brendle2}
n\left(\frac{(n+m)\left|\mathbb B^{n+m}\right|}{m\left|\mathbb B^m\right|}\right)^{\frac{1}{n}} \theta^{\frac{1}{n}}\left(\int_{\Sigma} f^{\frac{n}{n-1}}\right)^{\frac{n-1}{n}}
\le\int_{\Sigma} \sqrt{\left|\nabla^{\Sigma} f\right|^2+f^2|H|^2}+\int_{\partial \Sigma} f.
\end{equation}
Equality occurs if and only if $f$ is constant, $m=1$ or $2$, $M$ is the Euclidean space and $\Sigma$ is a flat round ball.

This result sharpens the classical Michael-Simon Sobolev inequality \cite{michael1973Sobolev}, and it reduces to the isoperimetric-type inequality when $f=1$. In particular, it confirms a longstanding conjecture asserting that the classical isoperimetric inequality in the Euclidean space remains valid on minimal submanifolds of codimension at most two in the Euclidean space.

On the other hand, the inequality \eqref{brendle2} is not sharp on any closed submanifold, even in the Euclidean space. It is therefore desirable to derive a Sobolev-type inequality which is attainable for closed submanifolds $\Sigma$.

In this regard, we are able to prove the following result.
\begin{theorem}\label{thm nonneg Sobolev}
Let $n, m \in \mathbb{N}$ and $(M, g)$ be a complete non-compact Riemannian manifold of dimension $n+m$ with nonnegative sectional curvature and asymptotic volume ratio $\theta>0$. Suppose $\Sigma$ is a compact $n$-dimensional submanifold immersed in $M$ (possibly with boundary $\partial \Sigma$) such that the mean curvature vector $H$ of $\Sigma$ is nowhere vanishing. Let $\beta\in \mathbb R$ and $f$ be a positive smooth function on $\Sigma$.
Then
\begin{equation}\label{nonneg Sobolev m ge 3}
\theta C_{n, m}\left(\int_{\Sigma} n f^\beta\right)^{n+1} \le n \left(\int_{\Sigma}\left(\left|\nabla^{\Sigma} f\right|+f|H|\right)+\int_{\partial \Sigma} f\right)^{n+1} \int_{\Sigma} \frac{f^{(n+1)(\beta-1)}}{|H|},
\end{equation}
where $C_{n, m}$ is given by \eqref{eq C nm}.
\end{theorem}

\begin{theorem}\label{thm nonneg Sobolev equality}
With the same assumptions and notations as in Theorem \ref{thm nonneg Sobolev} and suppose that $1\le m\le3$, the equality in \eqref{nonneg Sobolev m ge 3} holds if and only if
$f$ {and $|H|$ are constant}, $\Sigma$ is umbilical and with no boundary, and
$|\Sigma|=\theta |\mathbb S^n| n^n|H|^{-n}$.
\end{theorem}

The proofs of Theorems \ref{thm log sob ineq} and Theorem \ref{thm nonneg Sobolev} are inspired by the method developed in \cite{Brendle2023}, which involves constructing a suitable ``transport map'' from a subset of the normal bundle of $\Sigma$ into the ambient space $M$ by solving a certain linear elliptic equation with Neumann boundary condition. By estimating the Jacobian determinant of this map, we obtain the desired geometric inequalities. However, our analysis necessarily differs from Brendle's approach: his estimates are carried out so that sharpness is attained by the flat disk with boundary, whereas ours are carried out so that sharpness occurs precisely when the submanifold is closed and umbilical when $m\le 3$.


We also present several applications of Theorem \ref{thm nonneg Sobolev} in the Section \ref{sec some applications}. These include a sharp Sobolev inequality for submanifolds with constant $|\sigma|$, as well as geometric inequalities involving the integral $\int_\Sigma \frac{1}{|\sigma|}$; see Corollaries \ref{cor sharp classical Sobolev} and \ref{cor 2}. We further examine the case where $\Sigma$ is a complete non-compact immersed surface. Under a growth condition on $\left(\int_{B_r}|\sigma|\right)^3 \int_{B_r} \frac{1}{|\sigma|}$, our result yields information about the topology of $\Sigma$, in particular its Cohn-Vossen deficit; see Corollary \ref{cor3}.

The rest of the paper is organized as follows. In Section \ref{sec proof non compact fenchel willmore}, we establish two versions of the Fenchel-Willmore inequality for complete non-compact submanifolds $\Sigma$: one in the surface case and the other in higher dimensions (Theorems \ref{thm non-compact fenchel willmore} and \ref{thm noncopct willmore higher dim}), assuming the validity of Theorem \ref{thm log sob ineq}. In Sections \ref{sec proof thm log sob ineq} and \ref{sec proof thm log sob ineq equality case}, we prove the logarithmic Sobolev inequality (Theorem \ref{thm log sob ineq}) and characterize its equality case (Theorem \ref{thm log sob intro equality}), respectively. Section \ref{sec proof thm nonneg Sobolev} presents a sharp Sobolev-type inequality (Theorem \ref{thm nonneg Sobolev}) that is attained for closed umbilical submanifolds $\Sigma$. In Section \ref{sec some applications}, we discuss some applications of Theorem \ref{thm nonneg Sobolev}.
Finally, in Section \ref{sec pf thm nonneg Sobolev equality}, we prove the Theorem \ref{thm nonneg Sobolev equality}, the equality case of Theorem \ref{thm nonneg Sobolev}.

\begin{notation*}
Throughout the paper, $(M, g)$ denotes a complete, noncompact Riemannian manifold of dimension $n+m$, and $\Sigma \subset M$ is an immersed submanifold of dimension $n$, possibly with boundary. The Levi-Civita connection of $(M, g)$ is denoted by $\bar{\nabla }$. On $\Sigma$, we denote the Levi-Civita connection, Laplacian, and divergence by $\nabla^{\Sigma}, \Delta_{\Sigma}$, and $\operatorname{div}_{\Sigma}$, respectively. The second fundamental form of $\Sigma$, denoted by II, is a symmetric bilinear form on the tangent bundle of $\Sigma$ that takes values in the normal bundle $T^{\perp} \Sigma$. At a point $x \in \Sigma$, for tangent vector fields $X, Y$ and a normal vector field $V$, the second fundamental form satisfies $\langle\mathrm{II}(X, Y), V\rangle=\left\langle\bar \nabla_X Y, V\right\rangle$. The mean curvature vector and the normalized mean curvature vector of $\Sigma$ are defined by $H=\operatorname{tr}\mathrm{II}$ and $\sigma=\tfrac{1}{n}H$, respectively.
\end{notation*}

\section*{Acknowledgements}
The second named author is grateful to Ben Andrews, Robert McCann, Chao Xia, and Yong Wei for valuable discussions. He was supported by the University of Wollongong Early-Mid Career Researcher Enabling Grant and the UOW Advancement and Equity Grant Scheme for Research 2024.

\section{Proofs of non-compact Fenchel-Willmore inequalities}\label{sec proof non compact fenchel willmore}

In this section, we prove the two versions of Fenchel-Willmore inequality for complete non-compact $\Sigma$: Theorem \ref{thm non-compact fenchel willmore} and Theorem \ref{thm noncopct willmore higher dim}, assuming the validity of Theorem \ref{thm log sob ineq} (and hence of Theorem \ref{thm fenchel willmore}). We remark that while Theorem \ref{thm noncopct willmore higher dim} holds (with no assumption on the isoperimetric constant, as it is automatic) also for closed $\Sigma$, it is already included in Theorem \ref{thm fenchel willmore}.

\begin{proof}[Proof of Theorem \ref{thm non-compact fenchel willmore}]
Hartman \cite[Theorem 7.1]{Hartman1964} proved that for almost every $r$, the boundary $\partial B_r$ is a piecewise smooth, embedded closed curve. In particular, the length $|\partial B_r|$ is well-defined for almost every $r$. By \cite[Theorem 1]{White1987}, we have $\int_\Sigma K < \infty$, and $\Sigma$ is homeomorphic to $\bar{\Sigma} \setminus \{p_1, \ldots, p_l\}$, where $\bar{\Sigma}$ is a closed 2-manifold and $\{p_i\}_{i=1}^{l}$ is a finite subset of $\bar{\Sigma}$.
By \cite[Theorem A]{Shiohama1985}, it follows that
$$\lim_{r \rightarrow \infty} \frac{|\partial B_r|}{r} = 2\pi \chi(\Sigma) - \int_\Sigma K. $$
Thus, we can take a sequence $r_i \to \infty$ such that $|\partial B_{r_i}|$ is defined and
$$\lim_{i \to \infty} \frac{\int_{B_{r_i}} |\sigma|}{r_i} = \limsup_{r \to \infty} \frac{\int_{B_r} |\sigma|}{r} = C. $$
Then,
$$\lim_{i \to \infty} \frac{|\partial B_{r_i}|}{\int_{B_{r_i}} |\sigma|} = \frac{2\pi \chi(\Sigma) - \int_\Sigma K}{C}. $$
In view of Theorem \ref{thm fenchel willmore}, we conclude that
\begin{align*}
\int_\Sigma |\sigma|^2 = \lim_{i \to \infty} \int_{B_{r_i}} |\sigma|^2
& \ge \lim_{i \to \infty} \theta C_{2, m} e^{ -\frac{|\partial B_{r_i}|}{\int_{B_{r_i}} |\sigma|} } \\
& = \theta C_{2, m} e^{ -\frac{2\pi \chi(\Sigma) - \int_\Sigma K}{C} }.
\end{align*}
\end{proof}

\begin{remark}
\begin{enumerate}
\item
It is not hard to see that $\limsup_{r\to\infty}\frac{\int_{B_r}|\sigma|}{r}$ is independent of the choice of the center of $B_r$.
\item
By \cite[Theorem A]{Shiohama1985}, $C$ can also be expressed as $\limsup_{r\to \infty}\frac{\left|\partial B_r\right|}{2\left|B_r\right|} \int_{B_r}|\sigma|$.
\item
By Cohn-Vossen theorem \cite{CohnVossen1935}, the quantity $2 \pi \chi(\Sigma)-\int_{\Sigma} K$ that appears on the RHS of \eqref{ineq non-compact willmore} is non-negative.
\end{enumerate}
\end{remark}

We now turn to a Fenchel-Willmore inequality which holds for higher dimensional $\Sigma$.
Motivated by the Euclidean isoperimetric inequality, we define the isoperimetric constant as follows.
\begin{definition}\label{def isop const}
Let $N$ be an $n$-dimensional Riemannian manifold.
The isoperimetric constant of $N$ is defined by
$$
C_{\mathrm{iso}}(N)=\inf_{\Omega} \frac{ |\partial \Omega| }{ |\Omega|^{\frac{n-1}{n}}}
$$
where $\Omega$ ranges over all $C^1$ open bounded subsets of $N$.
\end{definition}

For an $n$-dimensional ($n>1$) submanifold $\Sigma$ in $M$ with asymptotic volume ratio $\theta$, it is more natural to consider the quantity $\theta^{-\frac{1}{ n }} C_{\mathrm{iso}}(\Sigma)$. Intuitively, a small value of $C_{\mathrm{iso}}(\Sigma)$ means that some relatively large volumes can be enclosed by relatively small boundaries. For instance, when $\Sigma$ contains a relatively narrow ``bottleneck'' which encloses a large region, or when $\Sigma$ is simply a closed manifold (in which case $C_{\mathrm{iso}}=0$). Theorem \ref{thm noncopct willmore higher dim} says that if $\Sigma$ has a small isoperimetric constant, then the Fenchel-Willmore integral $\int_{\Sigma} |\sigma|^n$ admits a non-trivial lower bound. By Federer-Fleming theorem \cite[Theorem II.2.1]{Chavel2001}, the isoperimetric constant is also equal to the Sobolev constant $C_{\mathrm{S}}=\inf_f \frac{\|{\nabla } f\|_1}{\|f\|_{\frac{n}{n-1}}}$, where $f$ ranges over $C_c^{1}(N)$. Therefore, the constant $C_{\mathrm{iso}}$ in Theorem \ref{thm noncopct willmore higher dim} can also be replaced the Sobolev constant $C_{\mathrm{S}}$.

\begin{proof}[Proof of Theorem \ref{thm noncopct willmore higher dim}]
Take a sequence of bounded open sets $\Omega_i$ in $\Sigma$ such that
$\lim_{i\to \infty}\frac{|\partial \Omega_i|}{|\Omega_i|^{\frac{n-1}{n}}} =C_{\mathrm{iso}}(\Sigma)$
and $\left(\theta K_{n, m}\right)^{\frac{1}{n}} \frac{\left|\Omega_i\right|^{\frac{n-1}{n}}}{\left|\partial \Omega_i\right|}>\frac{1}{n}$ for all $i$.

By setting $f = 1$ in \eqref{brendle2}, we obtain, for all $i$,
\begin{align*}
0<\left(\theta K_{n, m}\right)^{\frac{1}{n}}
\frac{|\Omega_i|^{\frac{n-1}{n}}}{|\partial\Omega_i|} -\frac{1}{n} \le \frac{1}{|\partial\Omega_i|}\int_{\Omega_i}|\sigma|.
\end{align*}
Then by \eqref{ineq Fenchel-Willmore nonneg},
\begin{equation*}
\begin{aligned}
\int_{\Sigma} |\sigma|^n
\ge \int_{\Omega_i} |\sigma|^n
& \ge \theta C_{n, m} \exp\left(-\frac{|\partial \Omega_i|}{\int_{\Omega_i} |\sigma|} \right) \\
& \ge \theta C_{n, m} \exp\left(-\left((\theta K_{n, m})^{\frac{1}{n}} \frac{|\Omega_i|^{\frac{n-1}{n}}}{|\partial \Omega_i|} - \frac{1}{n} \right)^{-1} \right).
\end{aligned}
\end{equation*}

Letting $i\to \infty$ would give the result.
\end{proof}

\section{Proof of Theorem \ref{thm log sob ineq}}\label{sec proof thm log sob ineq}
In this section, we will prove inequality \eqref{log sob intro}. We are going to assume first $m\ge3$.

We will show that it suffices to prove the result in the case where $\Sigma$ is connected, by appealing to a simple algebraic inequality (Lemma \ref{lem alg ineq}). Accordingly, from now until the end of the proof of Theorem \ref{thm log sob ineq}, we assume that $\Sigma$ is connected. At the very end of the proof, we will return to address the general case where $\Sigma$ may have multiple connected components.

Let $f$ be a positive smooth function on $\Sigma$. Since the inequality is invariant under scaling $f$ by a positive constant, we may assume, without loss of generality, that $f$ satisfies the following normalization:
\begin{equation}\label{normalization log Sobolev}
\frac{n}{n+1} \int_{\Sigma} f | \sigma| \log f= \frac{1}{2 n} \int_{\Sigma} \frac{\left|\nabla^{\Sigma} f\right|^2}{f|{\sigma}|}+\int_{\partial \Sigma}f.
\end{equation}
Therefore, to prove \eqref{log sob intro}, it is equivalent to show that
\begin{equation}\label{(2.2)}
\theta(n+1) \frac{\left|\mathbb{S}^{n+m-1}\right|}{\left|\mathbb{S}^{m-2}\right|} \le \int_{\Sigma} f |\sigma|^n.
\end{equation}
For the given function $f$, let us consider the elliptic equation
\begin{equation}\label{eq: u}
\begin{cases}
\mathrm{div}\left(f \nabla^{\Sigma} u\right)= \frac{n}{n+1} f|\sigma| \log f-\frac{1}{2 n} \frac{\left|\nabla^{\Sigma} f\right|^2}{f|\sigma|}
\text { on } \Sigma\\
\langle \nabla^{\Sigma} u, \eta\rangle =1
\text { on } \partial\Sigma \text { if } \partial \Sigma \ne \emptyset.
\end{cases}
\end{equation}
Here, $\eta$ denotes the co-normal to $\partial \Sigma$.
Since $\Sigma$ is connected, the condition \eqref{normalization log Sobolev} ensures the existence of such a solution, which is unique up to an additive constant. By standard elliptic regularity theory, $u$ is $C^{2, \gamma}$ for any $0<\gamma<1$ (\cite[Theorem 6.30]{GilbargTrudinger}).

Denote by $T_x^{\perp} \Sigma$ the space of normal vectors at $x$ and $\widetilde{T}_x^{\perp} \Sigma:=\{V \in T_x^{\perp} \Sigma:\langle V, \sigma(x)\rangle=0\}$.
We define
\begin{equation}\label{def Omega U Ar}
\begin{aligned}
\Omega &:= \left\{ x \in \Sigma \setminus \partial \Sigma: \left| \nabla^{\Sigma} u(x) \right| < 1 \right\}, \\
U &:= \left\{ (x, y, t):
x \in \Sigma \setminus \partial \Sigma, \;
y \in \widetilde{T}_x^{\perp} \Sigma, \;
t \in \mathbb{R}
\text{ such that }
\left| \nabla^{\Sigma} u(x) \right|^2 + |y|^2 + t^2 < 1
\right\}, \\
A_r &:= \left\{ (x, y, t) \in U:
\begin{array}{l}
r u(z) + \tfrac{1}{2} d\big(z, \exp_x(r \nabla^{\Sigma} u(x) + r y + r t \tfrac{\sigma}{|\sigma|})\big)^2 \\
\quad \ge r u(x) + \tfrac{1}{2} r^2 \left(\left| \nabla^{\Sigma} u(x) \right|^2 + |y|^2 + t^2 \right) \\
\text{for all } z \in \Sigma
\end{array}
\right\}.
\end{aligned}
\end{equation}
Also define the map $\Phi_r: \widetilde {T}^{\perp} \Sigma \times \mathbb{R} \rightarrow M$ by
\begin{equation}\label{def Phi r}
\Phi_r(x, y, t)=\exp_x\left(r \left(\nabla^{\Sigma} u(x)+ y+ t \frac{\sigma(x)}{|\sigma(x)|}\right)\right).
\end{equation}
\begin{lemma}\label{2.1}
For every $x \in \Omega$, we have
$$
\Delta_{\Sigma} u(x) \le n|\sigma|\left(f(x)^{\frac{1}{n+1}}-\sqrt{1-\left|\nabla^{\Sigma} u(x)\right|^2}\right).
$$
\end{lemma}
\begin{proof}

For every point $x \in \Sigma$, the equation \eqref{eq: u} of $u$ implies
\begin{align*}
\Delta_{\Sigma} u & =\frac{n}{n+1} |\sigma|\log f-\frac{1}{2 n} \frac{\left|\nabla^{\Sigma} f\right|^2}{f^2|\sigma|}-\left\langle\frac{\nabla^{\Sigma} f}{f}, \nabla^{\Sigma} u\right\rangle \\
& =\frac{n}{n+1}|\sigma| \log f+\frac{n}{2}|\sigma|\left|\nabla^{\Sigma} u\right|^2-\frac{1}{2}\left|\frac{1}{\sqrt{n}} \frac{\nabla^{\Sigma} f}{f\sqrt{|\sigma|}}+\sqrt{n} \sqrt{|\sigma|}\nabla^{\Sigma} u\right|^2 \\
& \le \frac{n}{n+1} |\sigma|\log f+\frac{n}{2}|\sigma|\left|\nabla^{\Sigma} u\right|^2.
\end{align*}
Using the inequalities $\log \lambda \le \lambda-1$ for $\lambda>0$ and $\sqrt{1-\theta} \le 1-\frac{\theta}{2}$ for $0 \le \theta \le 1$, we obtain, for all $x \in \Omega$,
\begin{equation}\label{ineq Delta u1}
\frac{1}{n+1} \log f \le f^{\frac{1}{n+1}}-1 \text { and } \sqrt{1-\left|\nabla^{\Sigma} u\right|^2} \le 1-\frac{\left|\nabla^{\Sigma} u\right|^2}{2}.
\end{equation}
Therefore, for $x\in \Omega$,
\begin{equation}\label{ineq Delta u2}
\Delta_{\Sigma} u \le n|\sigma|\left(\left(f^{\frac{1}{n+1}}-1\right)+\frac{\left|\nabla^{\Sigma} u\right|^2}{2}\right) \le n|\sigma|\left(f^{\frac{1}{n+1}} -\sqrt{1-\left|\nabla^{\Sigma} u\right|^2}\right).
\end{equation}
\end{proof}

The following two lemmas are taken from \cite[Lemma 4.2, Lemma 4.4]{Brendle2023} respectively.
\begin{lemma}\label{lem 2.2}
For every $0 \le \alpha<1$ and $r>0$,
$$
\{p \in M: \alpha r<d(x, p)<r \text { for all } x \in \Sigma\}
$$
is contained in
$$
\left\{\Phi_r(x, y, t):(x, y, t) \in A_r \text { and }\left|\nabla^{\Sigma} u(x)\right|^2+|y|^2+t^2>\alpha^2\right\}.
$$
\end{lemma}
\begin{lemma}\label{lem 2.3}
For every $(x, y, t) \in A_r$, we have
$$
g_{\Sigma}(x)+r\left(D_{\Sigma}^2 u(x)-\langle\mathrm{II}(x), y\rangle-t\left\langle\mathrm{II}(x), \frac{\sigma}{|\sigma|}\right\rangle\right) \ge 0.
$$
\end{lemma}
\begin{lemma}\label{lem 2.4}
Let $(x, y, t) \in A_r$. We have
\begin{equation}\label{(2.5)}
1+r|\sigma|\left(f(x)^{\frac{1}{n+1}}-\sqrt{1-\left|\nabla^{\Sigma} u(x)\right|^2}-t\right) \ge 0
\end{equation}
and
\begin{equation}\label{2.4'}
\left|\det D \Phi_r(x, y, t)\right| \le r^m\left(1+r|\sigma(x)|\left(f(x)^{\frac{1}{n+1}}-\sqrt{1-\left|\nabla^{\Sigma} u(x)\right|^2}-t\right)\right)^n.
\end{equation}
\end{lemma}

\begin{proof}
Fix a point $(x, y, t) \in A_r$ and define $\mathrm{A}=D_{\Sigma}^2 u(x)-\langle\mathrm{II}(x), y\rangle-$ $t\langle\mathrm{II}(x), \frac{\sigma}{|\sigma|}\rangle$. We first show \eqref{(2.5)}. By Lemma \ref{2.1}, we have
\begin{equation}\label{(2.6)}
\operatorname{tr} \mathrm{A}=\Delta_{\Sigma} u(x)-n t|\sigma| \le n|\sigma|\left(f^{\frac{1}{n+1}}-\sqrt{1-\left|\nabla^{\Sigma} u\right|^2}-t\right).
\end{equation}
Since $g_{\Sigma}(x)+r \mathrm{A} \ge 0$ by Lemma \ref{lem 2.3}, we take its trace and apply \eqref{(2.6)} to obtain \eqref{(2.5)}:
$$
0 \le n+r \operatorname{tr} \mathrm{A} \le n+n r|\sigma|\left(f^{\frac{1}{n+1}} -\sqrt{1-\left|\nabla^{\Sigma} u\right|^2}-t\right).
$$

We now prove \eqref{2.4'}. We first claim that
the function
\begin{equation}\label{non-inc}
s \longmapsto \frac{\left|\det D\Phi_s(x, y, t)\right|}{s^m\left(1+s|\sigma|\left(f(x)^{\frac{1}{n+1}}-\sqrt{1-\left|\nabla^{\Sigma} u(x)\right|^2}-t\right)\right)^n}
\end{equation}
is non-increasing on ($0, r$).

Choose a positively oriented local orthonormal frame in a neighborhood of $x$ such that $\{e_i\}_{i=1}^{n} \subset T_x \Sigma$, $\{\nu_\alpha\}_{\alpha=n+1}^{n+m} \subset T_x^{\perp} \Sigma$, and $\left\langle\bar{\nabla}_{e_i} \nu_\alpha, \nu_\beta\right\rangle = 0$ at the point $x$. We define the geodesic $\gamma(s) = \exp_x\left(s\left(\nabla^{\Sigma} u(x) + y + t \frac{\sigma}{|\sigma|}\right)\right)$ for $s \in [0, r]$, and denote by $E_i(s)$ and $N_\alpha(s)$ the parallel transports of $e_i$ and $\nu_\alpha$, respectively, along $\gamma$.

Define the Jacobi fields $X_i(s)$ and $Y_\alpha(s)$ along $\gamma$, determined by the initial conditions:
$$
\begin{cases}
X_i(0) = e_i \\
\langle D_s X_i(0), e_j \rangle = \mathrm{A}(e_i, e_j) \\
\langle D_s X_i(0), \nu_\beta \rangle = \langle \mathrm{II}(e_i, \nabla^\Sigma u(x)), \nu_\beta \rangle
\end{cases}
\quad \text{and} \quad
\begin{cases}
Y_\alpha(0) = 0 \\
D_s Y_\alpha(0) = \nu_\alpha.
\end{cases}
$$

Let $P(s)$ be the $(n+m) \times (n+m)$ matrix defined on $[0, r]$ by
\begin{align*}
P_{i j}(s) & =\left\langle X_i(s), E_j(s)\right\rangle, & & P_{i \beta}(s)=\left\langle X_i(s), N_\beta(s)\right\rangle, \\
P_{\alpha j}(s) & =\left\langle Y_\alpha(s), E_j(s)\right\rangle, & & P_{\alpha \beta}(s)=\left\langle Y_\alpha(s), N_\beta(s)\right\rangle.
\end{align*}

By the argument in \cite[Proposition 4.6]{Brendle2022}, it follows that $\det P(s) > 0$ for all $s \in (0, r)$, $\left|\det D\Phi_s(x, y, t)\right| = \det P(s)$, and
\begin{equation}\label{P0}
\lim_{s \rightarrow 0^{+}} s^{-m} \det P(s) = 1
\end{equation}
for every $s \in (0, r)$. Moreover, the matrix $Q(s):= P(s)^{-1} P^{\prime}(s)$
satisfies
$$
\operatorname{tr} Q(s) \le \frac{m}{s}+\sum_{i=1}^n \frac{\lambda_i}{1+s \lambda_i}
$$
for $s \in(0, r)$, where $\lambda_1, \ldots, \lambda_n$ are the eigenvalues of $\mathrm{A}$ w.r.t. $g_\Sigma$. By Lemma \ref{lem 2.3}, $1+s\lambda_i>0$ for $s\in (0, r)$. So by the concavity and monotonicity of the function $\frac{z}{1+z}$ for $z>-1$ and the estimate \eqref{(2.6)} for $\operatorname{tr} \mathrm{A}$, we have
\begin{equation}\label{(2.8)}
\begin{aligned}
\operatorname{tr} Q(s) \le \frac{m}{s}+\sum_{i=1}^n \frac{\lambda_i}{1+s \lambda_i}
\le& \frac{m}{s}+\frac{\sum_{i=1}^{n}\lambda_i}{1+\frac{s}{n} \sum_{i=1}^{n}\lambda_i}\\
=& \frac{m}{s}+\frac{\operatorname{tr} \mathrm{A}}{1+\frac{s}{n} \operatorname{tr} \mathrm{A}}\\
\le& \frac{m}{s}+\frac{n|\sigma|\left(f^{\frac{1}{n+1}}-\sqrt{1-\left|\nabla^{\Sigma} u\right|^2}-t\right)}{1+{s}|\sigma|\left(f^{\frac{1}{n+1}}-\sqrt{1-\left|\nabla^{\Sigma} u\right|^2}-t\right)}.
\end{aligned}
\end{equation}

Since $\left|\det \Phi_s(x, y, t)\right|=\det P(s)$ for $s\in\left(0, r\right)$ and $\frac{d}{d s} \log \det P(s)=\operatorname{tr} Q(s)$, it follows that
$$
\frac{d}{d s} \left(\frac{\det P(s)}{s^m\left(1+s|\sigma|(f^{\frac{1}{n+1}}-\sqrt{1-\left|\nabla^{\Sigma} u\right|^2}-t)\right)^n}\right) \le 0. \quad \text { on }(0, r).
$$
This proves \eqref{non-inc}.
Together with the fact that
$\lim_{s \rightarrow 0^{+}} s^{-m}\left|\det D \Phi_s(x, y, t)\right|=1$, we conclude that
$$
\left|\det D \Phi_r(x, y, t)\right| \le r^m\left(1+r|\sigma(x)|\left(f(x)^{\frac{1}{n+1}}-\sqrt{1-\left|\nabla^{\Sigma} u(x)\right|^2}-t\right)\right)^n.
$$
\end{proof}

\begin{lemma}\label{lem alg ineq}
Let $1<l\in \mathbb N$. If $a_i>0$ and $A_i>0$ for $i=1, \cdots, l$. Let $h$ be a strictly increasing concave function on $(0, \infty)$, then
$$\sum_{i=1}^{l}a_ih (A_i)<\left(\sum_{j=1}^{l}a_j\right) h\left(\sum_{i=1}^{l}A_i\right). $$
In paricular, this holds for the functions $\log x$ and $x^{\frac{1}{n+1}}$.
\end{lemma}

\begin{proof}
Let $w_i=\frac{a_i}{\sum_{j=1}^l a_j}$. Then by concavity and monotonicity of $h$,
$$
\sum_{i=1}^l w_i h (A_i)\le h \left(\sum_{i=1}^l w_i A_i\right)<h \left(\sum_{i=1}^l A_i\right).
$$
Multiplying both sides by $\sum_{j=1}^l a_j$ then gives the result.
\end{proof}

We are now ready to prove Theorem \ref{thm log sob ineq}.
\begin{proof}[Proof of Theorem \ref{thm log sob ineq}]
From \eqref{(2.5)}, we note that for $(x, y, t) \in A_r$, then $t$ satisfies
\begin{equation}\label{ineq range t}
-\sqrt{1-\left|\nabla^{\Sigma} u(x)\right|^2}<t \le f(x)^{\frac{1}{n+1}}-\sqrt{1-\left|\nabla^{\Sigma} u(x)\right|^2}+\frac{1}{r|\sigma|}.
\end{equation}
Therefore,
\begin{equation}\label{ineq 1.14}
\begin{aligned}
& |\{p \in M: \alpha r<d(x, p)<r \text { for all } x \in \Sigma\}| \\
\le& \int_{\Omega} \int_{-\sqrt{1-\left|\nabla^{\Sigma} u\right|^2}}^{f^{\frac{1}{n+1}}-\sqrt{1-\left|\nabla^{\Sigma} u\right|^2}+\frac{1}{r|\sigma|}} \int_{Y_{\alpha, x, t} }
\left|\det D \Phi_r(x, y, t)\right| 1_{A_r}(x, y, t) d y d t d \mathrm{vol}_{\Sigma} \\
\le& \int_{\Omega} \int_{-\sqrt{1-\left|\nabla^{\Sigma} u\right|^2}}^{f^{\frac{1}{n+1}}-\sqrt{1-\left|\nabla^{\Sigma} u\right|^2}+\frac{1}{r|\sigma|}} \int_{Y_{\alpha, x, t} }
r^m\left(1+r|\sigma|\left(f^{\frac{1}{n+1}}-\sqrt{1-\left|\nabla^{\Sigma} u\right|^2}-t\right)\right)^n d y d t d \mathrm{vol}_{\Sigma},
\end{aligned}
\end{equation}
where $Y_{\alpha, x, t}:= \left\{ y \in \widetilde{T}_x^{\perp} \Sigma: \alpha^2 < \left| \nabla^{\Sigma} u(x) \right|^2 + |y|^2 + t^2 < 1 \right\}$.

By the inequality $b^{\frac{m-1}{2}}-a^{\frac{m-1}{2}} \le \frac{m-1}{2}(b-a)$ for $0 \le a \le b<1$ and $m \ge 3$, for every $x \in \Omega$, we have
\begin{equation}\label{ineq 1-alpha^2}
\begin{aligned}
|Y_{\alpha, x, t} |
=& \left|\mathbb{B}^{m-1}\right|\left(\left(1-\left|\nabla^{\Sigma} u\right|^2-t^2\right)_{+}^{\frac{m-1}{2}}-\left(\alpha^2-\left|\nabla^{\Sigma} u\right|^2-t^2\right)_{+}^{\frac{m-1}{2}}\right) \\
\le& \frac{m-1}{2}\left|\mathbb{B}^{m-1}\right|\left(1-\alpha^2\right).
\end{aligned}
\end{equation}
Therefore,
\begin{equation}\label{ineq 1.14'}
\begin{aligned}
& |\{p \in M: \alpha r<d(x, p)<r \text { for all } x \in \Sigma\} |\\
\le & \frac{m-1}{2}\left|\mathbb{B}^{m-1}\right|\left(1-\alpha^2\right) \int_{\Omega} \int_{-\sqrt{1-\left|\nabla^{\Sigma} u\right|^2}}^{f^{\frac{1}{n+1}}-\sqrt{1-\left|\nabla^{\Sigma} u\right|^2}+\frac{1}{r|\sigma|}}
r^m\left(1+r|\sigma|\left(f^{\frac{1}{n+1}}-\sqrt{1-\left|\nabla^{\Sigma} u\right|^2}-t\right)\right)^n d t d \mathrm{vol}_{\Sigma} \\
=& \frac{m-1}{2}\left|\mathbb{B}^{m-1}\right|\left(1-\alpha^2\right) \int_{\Omega}\frac{r^{m-1}}{(n+1)|\sigma|}\left(1+r|\sigma| f^{\frac{1}{n+1}}\right)^{n+1}d \mathrm{vol}_{\Sigma}.
\end{aligned}
\end{equation}

Dividing the above inequality by $r^{n+m}$ and letting $r \rightarrow \infty$, we conclude that
$$
\theta\left|\mathbb{B}^{n+m}\right|\left(1-\alpha^{n+m}\right) \le \frac{m-1}{2(n+1)}\left|\mathbb{B}^{m-1}\right|\left(1-\alpha^2\right) \int_{\Omega} f(x)|\sigma(x)|^n d \mathrm{vol}_{\Sigma}(x).
$$

Finally, we divide the previous inequality by $1-\alpha$ and let $\alpha \rightarrow 1$ to obtain
\begin{equation}\label{(2.11)}
\theta(n+m)\left|\mathbb{B}^{n+m}\right|
\le \frac{(m-1)\left|\mathbb{B}^{m-1}\right|}{n+1} \int_{\Omega} f|\sigma|^n
\le \frac{(m-1)\left|\mathbb{B}^{m-1}\right|}{n+1} \int_{\Sigma} f|\sigma|^n.
\end{equation}
This is equivalent to \eqref{(2.2)}. Note also that when $m=3$, by the identity $\left|\mathbb{S}^{n+2}\right|=\frac{2 \pi}{n+1}\left|\mathbb{S}^n\right|$, the constant $(n+1) \frac{\left|\mathbb{S}^{n+m-1}\right|}{\left|\mathbb{S}^{m-2}\right|}$ is simply $|\mathbb S^n|$.

Now, we suppose that $\Sigma$ is disconnected. Since \eqref{log sob intro} holds on each individual component $\Sigma_i$ of $\Sigma$, we take the sum over them and use Lemma \ref{lem alg ineq} for $h=\log$, $a_i=\int_{\Sigma_i}f |\sigma|$ and $A_i=\int_{\Sigma_i}f |\sigma|^n$ to finish the proof. The resulting inequality is strict.

Let us now consider the case where $m=1$ or $2$. In this case, by taking the product of $M^{n+m}$ with $\mathbb{R}^{3-m}$, we can view $\Sigma$ as a codimension $3$ submanifold. The right-hand side of the inequality remains unchanged, so we only need to determine the constant after $\theta$.
In this case, the constant is given by $C_{n, 3}=|\mathbb S^n|$.
\end{proof}

\section{Proof of Theorem \ref{thm log sob intro equality}}\label{sec proof thm log sob ineq equality case}
In this section, we examine the equality case of \eqref{log sob intro} when $m\le 3$.

As explained in the last paragraph of the proof of Theorem \ref{thm log sob ineq}, we can assume $m = 3$.
From the proof of inequality \eqref{log sob intro}, we know that $\Sigma$ has only one component. As in the proof, by rescaling $f$, we can assume that the normalization condition \eqref{normalization log Sobolev} holds, and $u$ is as defined in \eqref{eq: u}. Moreover, we are going to use the same notation and definitions as in the proof of \eqref{log sob intro}. Therefore, \eqref{(2.2)} is now an equality, i.e. we have
\begin{equation}\label{(2.2')}
\theta(n+1) \frac{\left|\mathbb{S}^{n+m-1}\right|}{\left|\mathbb{S}^{m-2}\right|} =\int_{\Sigma} f|\sigma|^n.
\end{equation}
It follows from \eqref{(2.11)} that $\Omega$ is dense in $\Sigma$.
\begin{lemma}\label{lem xyt=1}
For every $r>0, x \in \Omega, y \in \widetilde{T}_x^{\perp} \Sigma$ and $t \in[-1, 1]$ satisfying $\left|\nabla^{\Sigma} u(x)\right|^2+$ $|y|^2+t^2=1$, we have
$$
\left|\det D \Phi_r(x, y, t)\right| \ge r^m\left(1+r|\sigma(x)|\left(f(x)^{\frac{1}{n+1}}-\sqrt{1-\left|\nabla^{\Sigma} u(x)\right|^2}-t\right)\right)^n.
$$
\end{lemma}

\begin{proof}
Assume on the contrary that there exists $x_{0}\in \Omega$, $y_{0}\in \widetilde{T}_{x}^{\perp}\Sigma$ and $t_{0}\in [-1, 1]$ that satisfy $|\nabla^\Sigma u(x_{0})|^{2}+|y_{0}|^{2}+t_{0}^{2}=1$, such that
$$
\left|\det D \Phi_{r_{0}}(x_{0}, y_{0}, t_{0})\right| < r_{0}^m\left(1+r_{0}|\sigma(x_{0})|\left(f(x_{0})^{\frac{1}{n+1}}-\sqrt{1-\left|\nabla^{\Sigma} u(x_{0})\right|^2}-t_{0}\right)\right)^n
$$
for some $r_{0}>0$. Since this is an open condition, we can without loss of generality assume that $y_0\ne0$. Then by continuity, there exists $\varepsilon\in(0, 1)$ and a neighborhood $V$ of $(x_{0}, y_{0}, t_{0})$ in $\widetilde T^\perp \Sigma\times \mathbb R$, such that
$$
\left|\det D \Phi_{r_{0}}(x, y, t)\right| < (1-\varepsilon)r_{0}^m\left(1+r_{0}|\sigma(x)|\left(f(x)^{\frac{1}{n+1}}-\sqrt{1-\left|\nabla^{\Sigma} u(x)\right|^2}-t\right)\right)^n \text{ on } V.
$$
It then follows from Lemma \ref{lem 2.4} that for every $r > r_{0}$,
$$
\left|\det D \Phi_{r}(x, y, t)\right| < (1-\varepsilon)r^m\left(1+ r |\sigma(x)| \left(f(x)^{\frac{1}{n+1}}-\sqrt{1-\left|\nabla^{\Sigma} u(x)\right|^2}-t\right)\right)^n \text{ on } V\cap A_{r}.
$$
Let $Y_\alpha=Y_{\alpha, x, t}:=\left\{y \in \widetilde{T}_x^{\perp} \Sigma: \alpha^2<\left|\nabla^{\Sigma} u(x)\right|^2+|y|^2+t^2<1\right\}$, regarded as a subset of $\widetilde T^\perp \Sigma\times \mathbb R$.
Consequently, by applying Lemma \ref{lem 2.2} and \eqref{ineq range t}, and following the reasoning in \eqref{ineq 1.14} and \eqref{ineq 1.14'}, we have
\begin{equation}\label{ineq est equality case}
\begin{aligned}
& |\{p \in M: \alpha r<d(x, p)<r \text { for all } x \in \Sigma\}| \\
\le& \int_{\Omega} \int_{-\sqrt{1-\left|\nabla^{\Sigma} u\right|^2}}^{f^{\frac{1}{n+1}}-\sqrt{1-\left|\nabla^{\Sigma} u\right|^2}+\frac{1}{r|\sigma|}} \int_{Y_\alpha}
\left|\det D \Phi_r(x, y, t)\right| 1_{A_r}(x, y, t) d y d t d \mathrm{vol}_{\Sigma} \\
\le& \int_{\Omega} \int_{-\sqrt{1-\left|\nabla^{\Sigma} u\right|^2}}^{f^{\frac{1}{n+1}}-\sqrt{1-\left|\nabla^{\Sigma} u\right|^2}+\frac{1}{r|\sigma|}} \int_{Y_\alpha}
\big(1-\varepsilon\cdot 1_{V}(x, y, t)\big)r^m\left(1+r|\sigma|\left(f^{\frac{1}{n+1}}-\sqrt{1-\left|\nabla^{\Sigma} u\right|^2}-t\right)\right)^n d y d t d \mathrm{vol}_{\Sigma}\\
\le& \frac{m-1}{2}\left|\mathbb{B}^{m-1}\right|\left(1-\alpha^2\right) \int_{\Omega} \int_{-\sqrt{1-\left|\nabla^{\Sigma} u\right|^2}}^{f^{\frac{1}{n+1}}-\sqrt{1-\left|\nabla^{\Sigma} u\right|^2}+\frac{1}{r|\sigma|}}
r^m\left(1+r|\sigma|\left(f^{\frac{1}{n+1}}-\sqrt{1-\left|\nabla^{\Sigma} u\right|^2}-t\right)\right)^n d t d \mathrm{vol}_{\Sigma}\\
& - \varepsilon \int_{\Omega} \int_{-\sqrt{1-\left|\nabla^{\Sigma} u\right|^2}}^{f^{\frac{1}{n+1}}-\sqrt{1-\left|\nabla^{\Sigma} u\right|^2}+\frac{1}{r|\sigma|}} \int_{Y_\alpha}
1_{V}(x, y, t) r^m\left(1+r|\sigma|\left(f^{\frac{1}{n+1}}-\sqrt{1-\left|\nabla^{\Sigma} u\right|^2}-t\right)\right)^n d y d t d \mathrm{vol}_{\Sigma}\\
=&:J(\alpha, r)-\varepsilon I(\alpha, r)
\end{aligned}
\end{equation}
for all $r>r_{0}$.

We know from \eqref{(2.11)} that
\begin{equation}\label{ineq J}
\lim_{\alpha\to1}\lim_{r\to\infty}\frac{1}{1-\alpha}\cdot\frac{1}{r^{n+m}}J(\alpha, r)\le \frac{(m-1)\left|\mathbb{B}^{m-1}\right|}{n+1} \int_{\Sigma} f|\sigma|^n.
\end{equation}
We now estimate $\lim_{\alpha \rightarrow 1} \lim_{r \rightarrow \infty} \frac{1}{1-\alpha} \cdot \frac{1}{r^{n+m}}I(\alpha, r)$ from below. As $V$ is open, for $\alpha$ close enough to $1$,
$V\cap Y_{\alpha, x_0, t_0}$ contains an open set
\begin{align*}
\{(\rho, \theta): \alpha^2-|\nabla^{\Sigma}u(x_0)|^2-t_0^2 <\rho^2<1-|\nabla^{\Sigma}u(x_0)|^2-t_0^2, \theta \in \mathcal O \}
\end{align*}
in polar coordinates of $\widetilde T_{x_0}^\perp \Sigma$, where $\mathcal O\subset \mathbb{S}^{m-2}$ is an open set containing $\frac{y_0}{|y_0|}$.

It then follows that for $\alpha$ sufficiently close to $1$,
$$
|V\cap Y_{\alpha, x_0, t_0}|\ge\frac{|\mathcal O|}{m-1}\left[\left(1-|\nabla^{\Sigma} u(x_0)|^2-t_0^2\right)^{\frac{m-1}{2}}-\left(\alpha^2-|\nabla^{\Sigma} u(x_0)|^2-t_0^2\right)^{\frac{m-1}{2}}\right]
$$
The function $\varphi(s):=\frac{1}{m-1}\left(s^2-\left|\nabla^{\Sigma} u\left(x_0\right)\right|^2-t_0^2\right)^{\frac{m-1}{2}}$ satisfies
$\varphi^{\prime}(1)>0$. It follows that for $\alpha$ close enough to $1$, $|V\cap Y_{\alpha, x_0, t_0}|\ge \frac{1}{2}|\mathcal O|\varphi'(1)(1-\alpha)=:2\delta_1(1-\alpha)$ for some $\delta_1>0$ which is independent of $\alpha$.
We may shrink $V$ if necessary to ensure that for all $(x, y, t) \in V$, $|V \cap Y_{\alpha, x, t}| \ge \delta_1 (1 - \alpha)$ holds for all $\alpha$ close to $1$.


From this, we deduce that for $\alpha$ close to $1$,
\begin{align}\label{ineq lower bound 2}
\frac{1}{1-\alpha} \cdot \frac{1}{r^{n+m}} I(\alpha, r)\ge \frac{\delta_1}{r^n}\int_{B_\rho (x_0) }\int_{T_{x, r}\cap V}
\left(1+r|\sigma|\left(f^{\frac{1}{n+1}}-\sqrt{1-\left|\nabla^{\Sigma} u\right|^2}-t\right)\right)^n dt d\operatorname{vol}_{\Sigma}
\end{align}
for some $\rho>0$, where $T_{x, r}=
\left\{t\in \mathbb R:-\sqrt{1-\left|\nabla^{\Sigma} u(x)\right|^2}<t<f(x)^{\frac{1}{n+1}}-\sqrt{1-\left|\nabla^{\Sigma} u(x)\right|^2}+\frac{1}{r|\sigma(x)|}\right\}$.

Note that the integrand in \eqref{ineq lower bound 2} is non-negative and is decreasing in $t$, and so if we let $\tau(x)= f(x)^{\frac{1}{n+1}}-\sqrt{1-\left|\nabla^{\Sigma} u(x)\right|^2}$, we have
\begin{align*}
& \int_{T_{x_0, r} \cap V} \left(1 + r|\sigma(x_0)|\left(f(x_0)^{\frac{1}{n+1}} - \sqrt{1 - \left|\nabla^{\Sigma} u(x_0)\right|^2} - t\right)\right)^n dt \\
\ge & \int_{\tau(x_0) + \frac{1}{r|\sigma(x_0)|} - \delta_2}^{\tau(x_0) + \frac{1}{r|\sigma(x_0)|}} \left(1 + r|\sigma(x_0)|\left(f(x_0)^{\frac{1}{n+1}} - \sqrt{1 - \left|\nabla^{\Sigma} u(x_0)\right|^2} - t\right)\right)^n dt \\
=& \frac{1}{r|\sigma(x_0)|(n+1)}\left(\left(1 + r|\sigma(x_0)| \delta_2\right)^{n+1} - 1\right),
\end{align*}
where $\delta_2>0$ is chosen such that $V\cap \{x=x_0, y=y_0\}$ contains the segment $(t_0-\delta_2, t_0+\delta_2)$.
So by continuity,
\begin{align*}
\int_{T_{x, r} \cap V}\left(1+r|\sigma|\left(f^{\frac{1}{n+1}}-\sqrt{1-\left|\nabla^{\Sigma} u\right|^2}-t\right)\right)^n d t
\ge& \frac{1}{2r\left|\sigma\left(x_0\right)\right|(n+1)}\left(\left(1+r\left|\sigma\left(x_0\right)\right| \delta_2\right)^{n+1}-1\right)\\
=&
\frac{ \delta_2^{n+1}r^n}{2(n+1)} |\sigma(x_0)|^n+o\left(r^n\right)
\end{align*}
for $x$ near $x_0$, where $o(r)$ is a quantity such that $\frac{o\left(r^n\right)}{r^n} \rightarrow 0$ as $r \rightarrow \infty$.

From this and \eqref{ineq lower bound 2}, it is then not hard to see that for some $C>0$,
\begin{align*}
\lim_{\alpha \rightarrow 1} \lim_{r \rightarrow \infty} \frac{1}{1-\alpha} \cdot \frac{1}{r^{n+m}}I(\alpha, r)\ge C >0.
\end{align*}

In view of \eqref{ineq est equality case} and \eqref{ineq J}, we obtain
\begin{align*}
\theta(n+m)\left|\mathbb{B}^{n+m}\right|\le \frac{(m-1)\left|\mathbb{B}^{m-1}\right|}{n+1} \int_{\Sigma} f|\sigma|^n-C\varepsilon.
\end{align*}
This means \eqref{(2.2')} is a strict inequality, a contradiction.
\end{proof}

Now we are ready to prove the equality case of Theorem \ref{thm log sob ineq}.

\begin{proof}[Proof of Theorem \ref{thm log sob intro equality}]
Fix $(x, y, t)$ such that $|\nabla^\Sigma u(x)|^{2}+|y|^{2}+t^{2}=1$.

Define $\mathrm{A}=D_{\Sigma}^2 u(x)-\langle\mathrm{II}(x), y\rangle-t\langle\mathrm{II}(x), \frac{\sigma}{|\sigma|}\rangle$. There exists small enough $s_{0}>0$, such that $g_{\Sigma}+sA > 0$ for all $0 < s < s_{0}$.
We may then define the vectors $\{e_1, \dots, e_n, \nu_1, \dots, \nu_m\}$ and hence the Jacobian matrix $P(s)$ in a manner analogous to their construction in the proof of Lemma \ref{lem 2.4}.

We have $|\det D\Phi_{s}(x, y, t)|=|\det P(s)|\ge s^m\left(1+s|\sigma(x)|\left(f(x)^{\frac{1}{n+1}}-\sqrt{1-\left|\nabla^{\Sigma} u(x)\right|^2}-t\right)\right)^n$ by Lemma \ref{lem xyt=1}, and $P(s)>0$ for small enough $s$.
Hence
\begin{align*}
\det P(s)\ge s^m\left(1+s|\sigma(x)|\left(f(x)^{\frac{1}{n+1}}-\sqrt{1-\left|\nabla^{\Sigma} u(x)\right|^2}-t\right)\right)^n>0
\end{align*}
for $s\in (0, s_0)$, by making $s_0$ smaller if necessary.

In particular, as in the proof of Lemma \ref{lem 2.4}, we can define $Q(s) = P(s)^{-1} P'(s)$ for $s\in (0, s_{0})$, which is symmetric.
Let $\lambda_i$ be the eigenvalues of $\mathrm{A}$. Note that $1+s\lambda_i>0$ for $s\in (0, s_0)$, and so the same computation as in \eqref{(2.8)} gives
\begin{equation*}
\begin{aligned}
\operatorname{tr} Q(s)
\le \frac{m}{s}+\sum_{i=1}^n \frac{\lambda_i}{1+s \lambda_i}
\le& \frac{m}{s}+\frac{\operatorname{tr} \mathrm{A}}{1+\frac{s}{n} \operatorname{tr} \mathrm{A}}\\
\le& \frac{m}{s}+\frac{n|\sigma|\left(f^{\frac{1}{n+1}}-\sqrt{1-\left|\nabla^{\Sigma} u\right|^2}-t\right)}{1+{s}|\sigma|\left(f^{\frac{1}{n+1}}-\sqrt{1-\left|\nabla^{\Sigma} u\right|^2}-t\right)},
\end{aligned}
\end{equation*}
which in turn gives
$$
\det P(s) \le s^m\left(1+s|\sigma|(f^{\frac{1}{n+1}}-\sqrt{1-\left|\nabla^{\Sigma} u\right|^2}-t)\right)^n,
$$
as in the proof of Lemma \ref{lem 2.4}.

Combining with Lemma \ref{lem xyt=1}, we obtain
$$
\det P(s) = s^m\left(1+s|\sigma|(f^{\frac{1}{n+1}}-\sqrt{1-\left|\nabla^{\Sigma} u\right|^2}-t)\right)^n
$$
for $s \in (0, s_{0})$. Therefore,
\begin{equation*}
\begin{aligned}
\operatorname{tr} Q(s) = \frac{m}{s}+\sum_{i=1}^n \frac{\lambda_i}{1+s \lambda_i} = \frac{m}{s}+\frac{\operatorname{tr} \mathrm{A}}{1+\frac{s}{n} \operatorname{tr} \mathrm{A}} = \frac{m}{s}+\frac{n|\sigma|\left(f^{\frac{1}{n+1}}-\sqrt{1-\left|\nabla^{\Sigma} u\right|^2}-t\right)}{1+{s}|\sigma|\left(f^{\frac{1}{n+1}}-\sqrt{1-\left|\nabla^{\Sigma} u\right|^2}-t\right)}.
\end{aligned}
\end{equation*}
From the second equality, we deduce that all eigenvalues of $\mathrm{A}$ have the same value $|\sigma|\left(f^{\frac{1}{n+1}}-\sqrt{1-\left|\nabla^{\Sigma} u\right|^2}-t\right)$, and the third equality implies that $f=1$ and $\nabla^{\Sigma}u=0$ from the proof of Lemma \ref{2.1}. In addition, $\Omega$ is a dense open set in $\Sigma$ from \eqref{(2.11)}. Therefore, $f \equiv 1$ and $u$ is constant in $\Sigma$. It follows that $\partial \Sigma=\emptyset$ as the boundary condition in \eqref{eq: u} cannot be satisfied otherwise. It also follows that $\Sigma$ is umbilical since the eigenvalues of $\mathrm{A}=-\langle\mathrm{II}(x), y\rangle-t\langle\mathrm{II}(x), \frac{\sigma(x)}{|\sigma(x)|}\rangle$ are $-t |\sigma(x)|$ for arbitrary $(y, t)$ satisfying $|y|^{2}+t^{2}=1$.
\end{proof}
\begin{remark}
From the above proof, we see that $f$ is constant and $\partial \Sigma=\emptyset$. Therefore, the equality case is reduced to
$$\int_{\Sigma} |\sigma|^n = \theta \, |\mathbb{S}^n|, $$
which means that the equality case of \eqref{ineq AFM} is attained. Moreover, by \cite[Theorem~1.2]{JiKwong2025}, one can further characterize the metric on the image under the normal exponential map of $\{(x, z) \in T^{\perp}\Sigma: \langle z, \sigma(x)\rangle \le 0 \}$.
For brevity, we do not reproduce the full statement here, and refer the interested reader to \cite{JiKwong2025}.
\end{remark}

\section{Proof of Theorem \ref{thm nonneg Sobolev}}\label{sec proof thm nonneg Sobolev}
In this section, we present the proof of Theorem \ref{thm nonneg Sobolev}. As the argument closely parallels that of Theorem \ref{thm log sob ineq}, we omit certain repetitive details for brevity.

Assume first $\Sigma$ is connected. Observe that the inequality is invariant under a rescaling of the Riemannian metric $g$. Therefore, we can rescale $g$ such that
\begin{equation}\label{normalization3}
\int_{\Sigma}\left(\left|\nabla^{\Sigma} f\right|+f|H|\right)+\int_{\partial \Sigma} f=\int_{\Sigma} n f^\beta.
\end{equation}

Denote by $\eta$ the co-normal to $\partial \Sigma$. For the given function $f$, let $u$ to be the solution of the problem
\begin{align*}
\operatorname{div}_{\Sigma}\left(f \nabla^{\Sigma} u\right)=n f^\beta-\left|\nabla^{\Sigma} f\right|-f|H| \text { on } \Sigma \\
\left\langle\nabla^{\Sigma} u, \eta\right\rangle=1 \quad \text { on } \partial \Sigma \text { if } \partial \Sigma \ne \emptyset.
\end{align*}




For each $x \in \Sigma$, we define $T_x^\perp \Sigma$ and $\widetilde{T}_x^\perp \Sigma$ as before. We also define the sets $\Omega$, $U$, and $A_r$, as well as the map $\Phi_r$, exactly as in the proof of Theorem \ref{thm log sob ineq}; see equations \eqref{def Omega U Ar} and \eqref{def Phi r}.

The following two lemmas, taken from \cite[Lemma 4.2, Lemma 4.4]{Brendle2023} respectively, continue to hold.
\begin{lemma}\label{lem nonneg A}
For every $0 \le \alpha<1$ and $r>0$, the set
$$
\{p \in M: \alpha r<d(x, p)<r \text { for all } x \in \Sigma\}
$$
is contained in
$$
\left\{\Phi_r(x, y, t):(x, y, t) \in A_r \text { and }\left|\nabla^{\Sigma} u(x)\right|^2+|y|^2+t^2>\alpha^2\right\}.
$$
\end{lemma}
\begin{lemma}\label{lem 2.3'}
For every $(x, y, t) \in A_r$, we have
$$
g_{\Sigma}(x)+r\left(D_{\Sigma}^2 u(x)-\langle\mathrm{II}(x), y\rangle-t\left\langle\mathrm{II}(x), \frac{H}{|H|}\right\rangle\right) \ge 0
$$
\end{lemma}

\begin{lemma}\label{lem 2.4'}
Let $(x, y, t) \in A_r$. We have
$$
1+ r\left(f^{\beta-1}-|\sigma|-t|\sigma|\right)\ge0.
$$
Moreover, the Jacobian determinant of $\Phi_r$ satisfies
\begin{align*}
\left|\det D \Phi_r(x, y, t)\right| \le r^m\left(1+r\left(f^{\beta-1}-|\sigma|-t|\sigma|\right)\right)^n.
\end{align*}
\end{lemma}
\begin{proof}
Let us fix a point $(x, y, t) \in A_r$ and denote $\mathrm{A}=D_{\Sigma}^2 u(x)-\langle\mathrm{II}(x), y\rangle-$ $t\langle\mathrm{II}(x), \frac{H}{|H|}\rangle$. Then
\begin{equation*}
\begin{aligned}
\text{tr }\mathrm{A}=\Delta u-t|H| & =\frac{1}{f} \operatorname{div}(f \nabla^\Sigma u)-\frac{1}{f}\left\langle\nabla^{\Sigma} f, \nabla^\Sigma u\right\rangle-t|H| \\
& =n f^{\beta-1}-\frac{1}{f}\left|\nabla^{\Sigma} f\right|-|H|-\frac{1}{f}\left\langle\nabla^{\Sigma} f, \nabla^\Sigma u\right\rangle-t|H| \\
& \le n f^{\beta-1}-|H|-t|H| \\
& =n\left(f^{\beta-1}-|\sigma|-t|\sigma|\right).
\end{aligned}
\end{equation*}
Hence, by Lemma \ref{lem 2.3'},
\begin{equation}\label{(2.5')}
0 \le n+r \operatorname{tr} \mathrm{A} \le n+n r\left(f^{\beta-1}-|\sigma|-t|\sigma|\right).
\end{equation}

We define $\{e_1, \cdots, e_n, \nu_{n+1}, \cdots, \nu_{n+m}\}$ and the square matrix $P(s)$ of size $(n+m)$ for $s\in [0, r]$ as in the proof of Lemma \ref{lem 2.4}, and let us recycle the notations in that proof. Then as in the proof of Lemma \ref{lem 2.4},
the matrix $Q(s):=P(s)^{-1} P^{\prime}(s)$ is symmetric for each
for $s \in(0, r)$ and its trace satisfies
\begin{equation*}
\begin{aligned}
\operatorname{tr} Q(s) \le& \frac{m}{s}+\sum_{i=1}^n \frac{\lambda_i}{1+s \lambda_i}\\
\le & \frac{m}{s}+\frac{\sum_{i=1}^n \lambda_i}{1+\frac{s}{n} \sum_{i=1}^n \lambda_i} \\
=& \frac{m}{s}+\frac{\operatorname{tr} \mathrm{A}}{1+\frac{s}{n} \operatorname{tr} \mathrm{A}}\\
\le & \frac{m}{s}+\frac{n\left(f^{\beta-1}-|\sigma|-t|\sigma|\right)}{1+s\left(f^{\beta-1}-|\sigma|-t|\sigma|\right)}.
\end{aligned}
\end{equation*}

Since $\frac{d}{d s} \log \det P(s)=\operatorname{tr} Q(s)$, it follows that
$$
\frac{d}{d s} \left(\frac{\left| \det \Phi_s(x, y, t)\right|}{s^m\left(1+s\left(f^{\beta-1}-|\sigma|-t|\sigma|\right)\right)^n}\right) \le 0 \quad \text { on }(0, r).
$$

As before, it then follows that for every $(x, y, t) \in A_r$, we have
$$
\left|\det D \Phi_r(x, y, t)\right| \le r^m\left(1+r\left(f^{\beta-1}-|\sigma|-t|\sigma|\right)\right)^n.
$$
\end{proof}

We are now ready to prove Theorem \ref{thm nonneg Sobolev}.
\begin{proof}[{Proof of Theorem \ref{thm nonneg Sobolev}}]
Assume first $m\ge 3$.
Notice that for $(x, y, t) \in A_r$, by \eqref{(2.5')}, $t$ satisfies
$$
-1\le -\sqrt{1-\left|\nabla^{\Sigma} u(x)\right|^2}<t \le \frac{f^{\beta-1}}{|\sigma|}-1+\frac{1}{r|\sigma|}.
$$
Therefore, by Lemma \ref{lem nonneg A} and Lemma \ref{lem 2.4'},
\begin{equation*}
\begin{aligned}
& |\{p \in M: \alpha r<d(x, p)<r \text { for all } x \in \Sigma\} | \\
\le& \int_{\Omega} \int_{-1}^{\frac{f^{\beta-1}}{|\sigma|}-1+\frac{1}{r|\sigma|}} \int_{Y_{\alpha, x, t} }
\left|\det D \Phi_r(x, y, t)\right| 1_{A_r}(x, y, t) d y d t d \operatorname{vol}_{\Sigma}(x) \\
\le& \int_{\Omega} \int_{-1}^{\frac{f^{\beta-1}}{|\sigma|}-1+\frac{1}{r|\sigma|}} \int_{Y_{\alpha, x, t} }
r^m\left(1+r\left(f^{\beta-1}-|\sigma|-t|\sigma|\right)\right)^n
d y d t d \operatorname{vol}_{\Sigma}(x).
\end{aligned}
\end{equation*}
where $Y_{\alpha, x, t}:=\left\{y \in \widetilde{T}_x^{\perp} \Sigma: \alpha^2<\left|\nabla^{\Sigma} u(x)\right|^2+|y|^2+t^2<1\right\}$.

As in \eqref{ineq 1-alpha^2}, $|Y_{\alpha, x, t}| \le \frac{m-1}{2}\left|\mathbb{B}^{m-1}\right|\left(1-\alpha^2\right)$. Therefore,
\begin{align*}
& |\{p \in M: \alpha r<d(x, p)<r \text { for all } x \in \Sigma\} | \\
\le & \frac{m-1}{2}\left|\mathbb{B}^{m-1}\right|\left(1-\alpha^2\right) \int_{\Omega} \int_{-1}^{\frac{f^{\beta-1}}{|\sigma|}-1+\frac{1}{r|\sigma|}}
r^m\left(1+r\left(f^{\beta-1}-|\sigma|-t|\sigma|\right)\right)^n d t d \operatorname{vol}_{\Sigma}(x)\\
=& \frac{m-1}{2}\left|\mathbb{B}^{m-1}\right|\left(1-\alpha^2\right) \int_\Omega \frac{r^{m-1}(1+r^{n+1}f^{(\beta-1)(n+1)})}{(n+1)|\sigma|} d\operatorname{vol}_{\Sigma}.
\end{align*}

Dividing the above inequality by $r^{n+m}$ and letting $r \rightarrow \infty$, we conclude that
$$
\theta\left|\mathbb{B}^{n+m}\right|\left(1-\alpha^{n+m}\right) \le \frac{m-1}{2(n+1)}\left|\mathbb{B}^{m-1}\right|\left(1-\alpha^2\right) \int_{\Omega} \frac{f^{(\beta-1)(n+1)}}{|\sigma|} d \operatorname{vol}_{\Sigma}(x).
$$

Finally, we divide the previous inequality by $1-\alpha$ and let $\alpha \rightarrow 1$ to obtain
\begin{align*}
\theta(n+m)\left|\mathbb{B}^{n+m}\right|
\le& \frac{m-1}{n+1}\left|\mathbb{B}^{m-1}\right| \int_{\Omega} \frac{f^{(\beta-1)(n+1)}}{|\sigma|} d \operatorname{vol}(x)\\
\le& \frac{m-1}{n+1}\left|\mathbb{B}^{m-1}\right| \int_{\Sigma} \frac{f^{(\beta-1)(n+1)}}{|\sigma|} d \operatorname{vol}(x).
\end{align*}
In view of the normalization condition
$$
\int_{\Sigma}\left(\left|\nabla^{\Sigma} f\right|+f|H|\right)+\int_{\partial \Sigma} f=\int_{\Sigma} n f^\beta,
$$
we thus arrive at inequality \eqref{nonneg Sobolev m ge 3}:
$$
\frac{\theta(n+1)\left|\mathbb{S}^{n+m-1}\right|}{\left|\mathbb{S}^{m-2}\right|}\left(\int_{\Sigma} n f^\beta\right)^{n+1}
\le n \left(\int_{\Sigma}\left(\left|\nabla^{\Sigma} f\right|+f|H|\right)+\int_{\partial \Sigma} f\right)^{n+1} \int_{\Sigma} \frac{f^{(n+1)(\beta-1)}}{|H|}.
$$
As before, when $m=3$, the LHS becomes simply $\theta|\mathbb S^n|\left(\int_{\Sigma} n f^\beta\right)^{n+1}$.

In general, suppose $\Sigma$ is disconnected. For each component $\Sigma_i$, we have
$$
\mathrm{C}_{n, m} \int_{\Sigma_i} n f^\beta \le \left(\int_{\Sigma_i}\left(\left|\nabla^{\Sigma} f\right|+f|H|\right)+\int_{\partial \Sigma_i} f\right)\left(\int_{\Sigma_i} \frac{f^{(n+1)(\beta-1)}}{|H|}\right)^{\frac{1}{n+1}}
$$
for some $\mathrm{C}_{n, m}>0$.

Summing these inequalities over all $i$ and applying Lemma \ref{lem alg ineq} to $h(z)=z^{\frac{1}{n+1}}$, $a_i=\int_{\Sigma_i}\left(\left|\nabla^{\Sigma_i} f\right|+f|H|\right)+\int_{\partial \Sigma_i} f$ and $A_i=\int_{\Sigma_i} \frac{f^{(n+1)(\beta-1)}}{|H|}$ to finish the proof. The resulting inequality is strict.

Let us now consider the case where $m=1$ or $2$. In this case, by taking the product of $M$ with $\mathbb{R}^{3-m}$, we can view $\Sigma$ as a codimension $3$ submanifold in a non-negatively curved manifold. The right-hand side of the inequality remains unchanged, so we only need to determine the constant in front of $\left(\int_{\Sigma} n f^\beta\right)^{n+1}$.
In this case, the constant is given by $\theta C_{n, 3}=\theta|\mathbb S^n|$.
\end{proof}

\section{Some applications of Theorem \ref{thm nonneg Sobolev}}\label{sec some applications}

By setting $\beta = \frac{n+1}{n}$ in Theorem \ref{thm nonneg Sobolev}, we recover the following sharp Sobolev inequality. The equality case is achieved, for instance, when $\Sigma = \mathbb{S}^n \subset \mathbb{R}^{n+m}$ is the standard round sphere and $m\le 3$.

\begin{corollary}\label{cor sharp classical Sobolev}
Let $n, m \in \mathbb{N}$, and let $(M, g)$ be a complete non-compact Riemannian manifold of dimension $n+m$ with nonnegative sectional curvature and asymptotic volume ratio $\theta > 0$. Suppose $\Sigma$ is a compact $n$-dimensional submanifold immersed in $M$ (possibly with boundary $\partial \Sigma$), and that the mean curvature vector $H$ satisfies $|H| = n$ on $\Sigma$. Then for any smooth positive function $f$ on $\Sigma$, the following inequality holds:
\begin{equation*}
\theta^{\frac{1}{n+1}} C_{n, m}^{\frac{1}{n+1}}
\left(\int_{\Sigma} f^{\frac{n+1}{n}}\right)^{\frac{n}{n+1}}
\le \int_{\Sigma} \left(f+\frac{|\nabla^{\Sigma} f|}{n} \right) +\frac{1}{n} \int_{\partial \Sigma} f,
\end{equation*}
where $C_{n, m}$ is given by \eqref{eq C nm}.
\end{corollary}

On the other hand, by setting $f=1$ in Theorem \ref{thm nonneg Sobolev}, we obtain the following corollary.
\begin{corollary}\label{cor 2}
Let $n, m \in \mathbb{N}$, and let $(M, g)$ be a complete non-compact Riemannian manifold of dimension $n+m$ with nonnegative sectional curvature and asymptotic volume ratio $\theta > 0$. Suppose $\Sigma$ is a compact $n$-dimensional submanifold immersed in $M$, possibly with boundary $\partial \Sigma$, such that the normalized mean curvature vector $\sigma$ is nowhere vanishing on $\Sigma$. Then
\begin{equation}\label{ineq cor2}
\theta C_{n, m} \le \left(\frac{1}{|\Sigma|}\int_{\Sigma} |\sigma| + \frac{| \partial \Sigma |}{n|\Sigma|} \right)^{n+1} \int_{\Sigma} \frac{1}{|\sigma|},
\end{equation}
where $C_{n, m}$ is given by \eqref{eq C nm}.
\end{corollary}
\begin{remark}
\begin{enumerate}
\item
In the case $n=m=1$, Corollary \ref{cor 2} gives the following Heintze-Karcher type inequality for geodesically convex curves.
Let $\Sigma$ be a connected, geodesically convex curve on a complete surface $M$ with nonnegative curvature $K$
and asymptotic volume ratio $\theta>0$, and suppose that $\Sigma$ bounds a domain $\Omega \subset M$. Then
$$
\int_\Sigma \frac{1}{\kappa}\, ds \;\ge\;
\frac{2\pi \theta\, L^2}{\bigl(2\pi - \int_\Omega K\, dA\bigr)^2},
$$
where $L = |\Sigma|$ is the length of $\Sigma$ and $\kappa$ is its geodesic curvature.

\item
When $\Sigma = \partial \Omega$ is the boundary of a smooth star-shaped domain in $\mathbb R^{n+1}$, inequality \eqref{ineq cor2} can be derived from the quermassintegral inequality in \cite[Theorem 2]{GL09} and the Cauchy-Schwarz inequality. This naturally leads to the question of whether a version of the quermassintegral inequality holds for hypersurfaces in $\mathbb R^{n+1}$ with boundary.
\end{enumerate}
\end{remark}

Similar to Theorem \ref{thm non-compact fenchel willmore}, we may derive certain consequences from Theorem \ref{thm nonneg Sobolev} (more specifically, Corollary \ref{cor 2}) for complete non-compact immersed surfaces by passing to the limit.
The result is an upper bound on the Cohn-Vossen deficit in terms of integrals involving the norm of the mean curvature and its reciprocal.

\begin{corollary}\label{cor3}
Let $\left(M^{2+m}, g\right)$ be a complete non-compact Riemannian manifold with nonnegative sectional curvature and asymptotic volume ratio $\theta>0$. Suppose that $\Sigma$ is a complete non-compact surface immersed in $M$ such that the mean curvature vector $\sigma$ of $\Sigma$ is nowhere vanishing. Assume the following conditions hold:
\begin{enumerate}
\item
$\liminf_{r \rightarrow \infty} \frac{1}{r^6} \left(\int_{B_r}|\sigma|\right)^3\int_{B_r}\frac{1}{|\sigma|}=C$, where $B_r$ is the metric ball of radius $r$ on $\Sigma$ with a fixed center.
\item The negative part of the Gaussian curvature of $\Sigma$ is $L^1$, i.e. $\int_{\Sigma} K^{-}<\infty$.
\end{enumerate}
Then
\begin{align*}
\theta C_{2, m}\left(2\pi\chi(\Sigma)-\int_\Sigma K\right)^3 \le 8C.
\end{align*}
In particular, if $\liminf_{r \rightarrow \infty} \frac{1}{r^6}\left(\int_{B_r}|\sigma|\right)^3 \int_{B_r} \frac{1}{|\sigma|}=0$, then the Cohn-Vossen deficit $2\pi\chi(\Sigma)-\int_\Sigma K=0$.
\end{corollary}
\begin{proof}
As in the proof of Theorem \ref{thm non-compact fenchel willmore}, we can take a sequence $r_i \rightarrow \infty$ such that $\left|\partial B_{r_i}\right|$ is defined and
$$
\lim_{i \rightarrow \infty} \frac{1}{{r_i}^6}\left(\int_{B_{r_i}}|\sigma|\right)^3 \int_{B_{r_i}} \frac{1}{|\sigma|} =C
$$
By \cite[Theorem A]{Shiohama1985},
$$
\lim_{i \rightarrow \infty} \frac{2|B_{r_i}|}{r_i^2}=2 \pi \chi(\Sigma)-\int_\Sigma K\quad \text{and}\quad
\lim_{i \rightarrow \infty} \frac{|\partial B_{r_i}|}{| B_{r_i}|}=0. $$
By applying Corollary \ref{cor 2} to $B_{r_i}$ and taking $i\to \infty$, we can get the result.

As a consequence, the Cohn-Vossen inequality shows that $C = 0$ would force the Cohn-Vossen deficit of $\Sigma$ to be zero.
\end{proof}

\section{Proof of Theorem \ref{thm nonneg Sobolev equality}}\label{sec pf thm nonneg Sobolev equality}
The idea is similar to the proof of the equality case of Theorem \ref{thm log sob ineq}. Let $1\le m\le 3$, and without loss of generality, we assume that
\begin{align*}
\int_{\Sigma}\left(\left|\nabla^{\Sigma} f\right|+f|H|\right)+\int_{\partial \Sigma} f=\int_{\Sigma} n f^\beta.
\end{align*}
Then the equality becomes
\begin{equation}\label{eq thm 5}
\frac{\theta(n+1)\left|\mathbb{S}^{n+m-1}\right|}{\left|\mathbb{S}^{m-2}\right|} = n \int_{\Sigma} \frac{f^{(n+1)(\beta-1)}}{|H|}.
\end{equation}
As in Theorem \ref{thm log sob intro equality}, $\Sigma$ is connected.
\begin{lemma}\label{lem xyt=1'}
For every $r>0, x \in \Omega, y \in \widetilde{T}_x^{\perp} \Sigma$ and $t \in[-1, 1]$ satisfying $\left|\nabla^{\Sigma} u(x)\right|^2+$ $|y|^2+t^2=1$, we have
\begin{align*}
\left|\det D \Phi_r(x, y, t)\right| \ge r^m\left(1+r\left(f^{\beta-1}-|\sigma|-t|\sigma|\right)\right)^n.
\end{align*}
\end{lemma}

\begin{proof}
The proof is similar to Lemma \ref{lem xyt=1}. Assume on the contrary that there exists $x_{0}\in \Omega$, $y_{0}\in \widetilde{T}_{x}^{\perp}\Sigma$ and $t_{0}\in [-1, 1]$ that satisfy $|\nabla^\Sigma u(x_{0})|^{2}+|y_{0}|^{2}+t_{0}^{2}=1$, such that
$$
\left|\det D \Phi_{r_{0}}(x_{0}, y_{0}, t_{0})\right| < r_{0}^m\left(1+r_{0}\left(f(x_{0})^{\beta-1}-|\sigma(x_{0})|-t|\sigma(x_{0})|\right)\right)^n
$$
for some $r_{0}>0$. By continuity, there exists $\varepsilon\in(0, 1)$ and a neighborhood $V$ of $(x_{0}, y_{0}, t_{0})$ in $\widetilde T^\perp \Sigma\times \mathbb R$, such that
$$
\left|\det D \Phi_{r_{0}}(x, y, t)\right| < (1-\varepsilon)r_{0}^m\left(1+r_{0}\left(f(x)^{\beta-1}-|\sigma(x)|-t|\sigma(x)|\right)\right)^n \text{ on } V.
$$
It then follows from Lemma \ref{lem 2.4'} that for every $r > r_{0}$,
$$
\left|\det D \Phi_{r}(x, y, t)\right| < (1-\varepsilon)r^m\left(1+r\left(f(x)^{\beta-1}-|\sigma(x)|-t|\sigma(x)|\right)\right)^n \text{ on } V\cap A_{r}.
$$
Consequently,
\begin{align*}
& |\{p \in M: \alpha r<d(x, p)<r \text { for all } x \in \Sigma\} | \\
\le& \int_{\Omega} \int_{-\sqrt{1-|\nabla^\Sigma u|^{2}}}^{\frac{f^{\beta-1}}{|\sigma|}-1+\frac{1}{r|\sigma|}} \int_{Y_{\alpha, x, t}}
\left|\det D \Phi_r(x, y, t)\right| 1_{A_r}(x, y, t) d y d t d \operatorname{vol}_{\Sigma}(x) \\
\le& \int_{\Omega} \int_{-1}^{\frac{f^{\beta-1}}{|\sigma|}-1+\frac{1}{r|\sigma|}} \int_{Y_{\alpha, x, t}}
\left(1-\varepsilon\cdot 1_{V}(x, y, t)\right)r^m\left(1+r\left(f(x)^{\beta-1}-|\sigma(x)|-t|\sigma(x)|\right)\right)^n
d y d t d \operatorname{vol}_{\Sigma}(x). \\
\le & \frac{m-1}{2}\left|\mathbb{B}^{m-1}\right|\left(1-\alpha^2\right) \int_{\Omega}\frac{r^{m-1}}{(n+1)|\sigma|}\left(1+r|\sigma| f^{\frac{1}{n+1}}\right)^{n+1}d \mathrm{vol}_{\Sigma} - \varepsilon I(\alpha, r),
\end{align*}
where $Y_{\alpha, x, t}=\left\{y \in \widetilde{T}_x^{\perp} \Sigma: \alpha^2<\left|\nabla^{\Sigma} u\right|^2+|y|^2+t^2<1\right\}$ and
$$
I(\alpha, r)=\int_{\Omega} \int_{-1}^{\frac{f^{\beta-1}}{|\sigma|}-1+\frac{1}{r|\sigma|}} \int_{Y_{\alpha, x, t}}
1_{V}(x, y, t)r^m\left(1+r\left(f(x)^{\beta-1}-|\sigma(x)|-t|\sigma(x)|\right)\right)^n
d y d t d \operatorname{vol}_{\Sigma}(x).
$$
As $r\to \infty$, we obtain
\begin{equation}\label{2.15'}
\theta\left|\mathbb{B}^{n+m}\right|\left(1-\alpha^{n+m}\right) \le \frac{m-1}{2(n+1)}\left|\mathbb{B}^{m-1}\right|\left(1-\alpha^2\right) \int_{\Omega} \frac{f^{(\beta-1)(n+1)}}{|\sigma|} d \operatorname{vol}_{\Sigma}(x) - \varepsilon\lim\limits_{r\to\infty}I(\alpha, r).
\end{equation}
By similar arguments as in the proof of Lemma \ref{lem xyt=1}, $\lim\limits_{\alpha\to 1}\lim\limits_{r\to \infty}\frac{1}{r^{n+m}}\frac{I(\alpha, r)}{1-\alpha}>0$. Therefore, dividing \ref{2.15'} by $1-\alpha$ and passing $\alpha\to 1$, we have
\begin{align*}
\theta(n+m)\left|\mathbb{B}^{n+m}\right|
<& \frac{m-1}{n+1}\left|\mathbb{B}^{m-1}\right| \int_{\Omega} \frac{f^{(\beta-1)(n+1)}}{|\sigma|} d \operatorname{vol}(x)\\
\le& \frac{m-1}{n+1}\left|\mathbb{B}^{m-1}\right| \int_{\Sigma} \frac{f^{(\beta-1)(n+1)}}{|\sigma|} d \operatorname{vol}(x),
\end{align*}
which contradicts to \eqref{eq thm 5}.
\end{proof}

\begin{proof}[Proof of Theorem \ref{thm nonneg Sobolev equality}]
Fix $(x, y, t)$ such that $|\nabla^\Sigma u(x)|^{2}+|y|^{2}+t^{2}=1$, by Lemma \ref{lem xyt=1'}, we have
$$
\left|\det D \Phi_r(x, y, t)\right| \ge r^m\left(1+r\left(f^{\beta-1}-|\sigma|-t|\sigma|\right)\right)^n.
$$
Let $s_{0}>0$ be small enough such that $g_{\Sigma}+sA>0$ for all $s\in (0, s_{0})$. Then,
\begin{align*}
\operatorname{tr} Q(s) \le \frac{m}{s}+\sum_{i=1}^n \frac{\lambda_i}{1+s \lambda_i} \le \frac{m}{s}+\frac{\operatorname{tr} \mathrm{A}}{1+\frac{s}{n} \operatorname{tr} \mathrm{A}} \le \frac{m}{s}+\frac{n\left(f^{\beta-1}-|\sigma|-t|\sigma|\right)}{1+s\left(f^{\beta-1}-|\sigma|-t|\sigma|\right)},
\end{align*}
which gives
\begin{align*}
\det P(s) = \det D\Phi_{s}(x, y, t) \le r^m\left(1+r\left(f^{\beta-1}-|\sigma|-t|\sigma|\right)\right)^n.
\end{align*}
Therefore,
\begin{align*}
\det P(s) = \det D\Phi_{s}(x, y, t) = r^m\left(1+r\left(f^{\beta-1}-|\sigma|-t|\sigma|\right)\right)^n,
\end{align*}
and hence
\begin{align*}
\operatorname{tr} Q(s) = \frac{m}{s}+\sum_{i=1}^n \frac{\lambda_i}{1+s \lambda_i} = \frac{m}{s}+\frac{\operatorname{tr} \mathrm{A}}{1+\frac{s}{n} \operatorname{tr} \mathrm{A}} = \frac{m}{s}+\frac{n\left(f^{\beta-1}-|\sigma|-t|\sigma|\right)}{1+s\left(f^{\beta-1}-|\sigma|-t|\sigma|\right)}.
\end{align*}
These equalities imply that $-\frac{1}{f}\left|\nabla^{\Sigma} f\right|-\frac{1}{f}\left\langle\nabla^{\Sigma} f, \nabla^\Sigma u\right\rangle=0$ from the proof of Lemma \ref{lem 2.4'} and all eigenvalues of $A$ are $f^{\beta-1}-|\sigma|-t|\sigma|$. Note that $|\nabla^\Sigma u|<1$ in $\Omega$, we have that $f$ is constant in $\Omega$, and therefore $f$ is constant in $\Sigma$ for $\Omega$ is dense in $\Sigma$. In addition, from the proof of Lemma \ref{2.4'}, the equality \eqref{2.15'} implies that $-1=-\sqrt{1-|\nabla^\Sigma u|^{2}}$ for the lower bound of $t$. Therefore, $\nabla^\Sigma u(x)=0$ and $\Sigma$ has no boundary, otherwise it contradicts the boundary condition $\left\langle\nabla^{\Sigma} u, \eta\right\rangle=1$ on $\partial \Sigma$. Recall that $u$ is the solution to
\begin{align*}
\operatorname{div}_{\Sigma}\left(f \nabla^{\Sigma} u\right)=n f^\beta-\left|\nabla^{\Sigma} f\right|-f|H| \text { on } \Sigma,
\end{align*}
it follows that $|H|=nf^{\beta-1}$ is also constant. Since
$$
A=D_{\Sigma}^2 u(x)-\langle\mathrm{II}(x), y\rangle-t\left\langle\mathrm{II}(x), \frac{H}{|H|}\right\rangle = \left(f(x)^{\beta-1}-|\sigma(x)|-t|\sigma(x)|\right)g_{\Sigma},
$$
for arbitrary $x, y, t$ satisfying $|\nabla^\Sigma u(x)|^{2}+|y|^{2}+t^{2}=1$,
we conclude that $\Sigma$ is umbilical. From this we obtain
\begin{align*}
\frac{|H|}{n}=\left(\frac{\theta|\mathbb S^n|}{|\Sigma|}\right)^{\frac{1}{n}}.
\end{align*}
\end{proof}


\begin{thebibliography}{10}
\bibitem{AFM20}
V. Agostiniani, M. Fogagnolo, and L. Mazzieri.
Sharp geometric inequalities for closed hypersurfaces in manifolds with nonnegative Ricci curvature.
\textit{Inventiones Mathematicae}, 222(3):1033--1101, 2020.


\bibitem{brendle2021isoperimetric}
S. Brendle.
\newblock The isoperimetric inequality for a minimal submanifold in Euclidean
space.
\newblock {\em Journal of the American Mathematical Society}, 34(2):595--603,
2021.

\bibitem{Brendle2022}
S. Brendle, The logarithmic Sobolev inequality for a submanifold in Euclidean space, \emph{Commun. Pure Appl. Math.} \textbf{75} (2022), 449--454.

\bibitem{Brendle2023}
S. Brendle, \emph{Sobolev inequalities in manifolds with nonnegative curvature},
\textit{Commun. Pure Appl. Math.} \textbf{76} (2023), 2192--2218.

\bibitem{BrendleEichmair2023}
S. Brendle and M. Eichmair.
\newblock Proof of the Michael--Simon--Sobolev inequality using optimal transport.
\newblock \emph{Journal für die reine und angewandte Mathematik (Crelle's Journal)}, 804:1-10, 2023.

\bibitem{BrendleHungWang2016}
S. Brendle, P.-K. Hung, and M.-T. Wang.
A Minkowski inequality for hypersurfaces in the Anti-de Sitter-Schwarzschild manifold.
\textit{Communications on Pure and Applied Mathematics}, 69(1):124--144, 2016.


\bibitem{cabre2008elliptic}
X. Cabr{\'e}.
\newblock Elliptic PDE's in probability and geometry: symmetry and regularity of solutions.
\newblock {\em Discrete \& Continuous Dynamical Systems}, 20(3):425, 2008.

\bibitem{Chavel2001}
I. Chavel,
\textit{Isoperimetric Inequalities: Differential Geometric and Analytic Perspectives},
Cambridge Tracts in Mathematics, vol. 145, Cambridge University Press, 2001.

\bibitem{Chen1971}
B.-Y. Chen.
\newblock On the total curvature of immersed manifolds. \ I: An inequality of Fenchel--Borsuk--Willmore.
\newblock \emph{American Journal of Mathematics}, \emph{93}(1), 148--162, 1971.

\bibitem{CohnVossen1935}
S. Cohn\textendash Vossen,
\emph{K\``urzeste Wege und Totalkr\''ummung auf Fl\"achen},
Compositio Math. \ \textbf{2} (1935), 69--133.

\bibitem{DelPinoDolbeault2003}
M. Del Pino, and J. Dolbeault. \emph{The optimal Euclidean $L^p$-Sobolev logarithmic inequality},
\textit{J. Funct. Anal.}, \textbf{197} (2003), no. 1, 151--161.

\bibitem{Ecker2000}
K. Ecker. \emph{Logarithmic Sobolev inequalities on submanifolds of Euclidean space},
\textit{J. Reine Angew. Math.}, \textbf{522} (2000), 105-118.

\bibitem{Fenchel29}
W. Fenchel, \``Uber Kr\''ummung und Windung geschlossener Raumkurven, \textit{Mathematische Annalen}, \textbf{101} (1929), 238--252.

\bibitem{GilbargTrudinger}
D. Gilbarg and N. S. Trudinger,
\textit{Elliptic Partial Differential Equations of Second Order},
2nd ed., Springer-Verlag, Berlin, 1983.

\bibitem{Gross1975}
L. Gross. \emph{Logarithmic Sobolev inequalities}, \textit{Amer. J. Math.}, \textbf{97} (1975), no. 4, 1061--1083.

\bibitem{GL09}
P. Guan and J. Li,
The quermassintegral inequalities for $k$-convex starshaped domains, \textit{Adv. Math.} \textbf{221} (2009),
1725--1732.

\bibitem{Hartman1964}
P. Hartman, \emph{Geodesic parallel coordinates in the large}, Amer. J. Math. \textbf{86} (1964), 705--727.

\bibitem{Huisken1990}
G. Huisken, \emph{Asymptotic behavior for singularities of the mean curvature flow},
\textit{J. Differential Geom.}, \textbf{31} (1990), no. 1, 285--299.

\bibitem{JiKwong2025}
M. Ji and K. K. Kwong,
\emph{Fenchel-Willmore inequality for submanifolds in manifolds with non-negative $k$-Ricci curvature},
\textit{arXiv:2507.07655 [math. DG]} (2025).

\bibitem{Ledoux2001}
M. Ledoux, \emph{Concentration of measure and logarithmic Sobolev inequalities},
\textit{Sém. Probab. XXXIII}, \textbf{1709} (2001), 120--216.

\bibitem{Ni2004}
L. Ni, \emph{The entropy formula for linear heat equation}, \textit{J. Geom. Anal.}, \textbf{14} (2004), no. 1, 87--100.

\bibitem{Perelman2002}
G. Perelman, \emph{The entropy formula for the Ricci flow and its geometric applications}, preprint, arXiv:math/0211159 (2002).

\bibitem{michael1973Sobolev}
J.H. Michael and L.M. Simon.
\newblock Sobolev and mean-value inequalities on generalized submanifolds of $\mathbb R^n$.
\textit {Commun. Pure Appl. Math.}, 26(3):361--379,
1973.

\bibitem{Pham25}
D. Pham.
A logarithmic Sobolev inequality for closed submanifolds with constant length of mean curvature vector.
\textit{Journal of Geometric Analysis}, 35:26, 2025.

\bibitem{Shiohama1985}
K. Shiohama, \emph{Total curvatures and minimal areas of complete open surfaces}, Proc. Amer. Math. Soc. \textbf{94} (1985), no. 2, 310--316.

\bibitem{Stam1959}
A. J. Stam, \emph{Some inequalities satisfied by the quantities of information of Fisher and Shannon},
\textit{Information and Control}, \textbf{2} (1959), 101--112.


\bibitem{Weissler1978}
F. B. Weissler. \emph{Logarithmic Sobolev inequalities for the heat-diffusion semigroup},
\textit{Trans. Amer. Math. Soc.}, \textbf{237} (1978), 255--269.

\bibitem{White1987}
B. White, \emph{Complete surfaces of finite total curvature}, J. Differential Geom. \textbf{26} (1987), no. 2, 315--326.

\bibitem{Willmore68}
T. J. Willmore.
Mean curvature of immersed surfaces.
\textit{An. Stiint. Univ. ``Al. I. Cuza'' Iasi Sect. I a Mat. (N.S.)}, 14:99--103, 1968.

\end{thebibliography}
\end{document}